\documentclass{article}
\usepackage{amsthm}
\usepackage{amsmath}
\usepackage{amssymb}
\usepackage{graphicx}

\usepackage{color,soul}

\title{Degenerate behavior in non-hyperbolic semigroup actions on the interval: fast growth of periodic points and universal dynamics}

\author{Masayuki Asaoka, Katsutoshi Shinohara and Dmitry Turaev}

\def\RR{\mathbb{R}}
\def\ZZ{\mathbb{Z}}

\def\cD{{\mathcal D}}
\def\cE{{\mathcal E}}

\def\cN{{\mathcal N}}
\def\cO{{\mathcal O}}
\def\cW{{\mathcal W}}

\def\cU{{\mathcal U}}
\def\cV{{\mathcal V}}

\def\cA{{\mathcal A}}

\def\ra{{\rightarrow}}

\def\vphi{\varphi}
\def\del{\partial}

\DeclareMathOperator{\Diff}{Diff}
\DeclareMathOperator{\Fix}{Fix}

\DeclareMathOperator{\Int}{Int}

\DeclareMathOperator{\supp}{supp}
\DeclareMathOperator{\sgn}{sgn}

\newcommand\ub[1]{{\underline{#1}}}

\theoremstyle{plain}
\newtheorem{thm}{Theorem}[section]
\newtheorem{prop}[thm]{Proposition}
\newtheorem{lemma}[thm]{Lemma}
\newtheorem{cor}[thm]{Corollary}

\theoremstyle{definition}

\theoremstyle{remark}
\newtheorem{rmk}[thm]{Remark}

\begin{document}

\maketitle

\begin{abstract}
We consider semigroup actions on the unit interval generated
by strictly increasing $C^r$-maps. We assume that one of the generators 
has a pair of fixed points, one attracting and one repelling, and a heteroclinic orbit that connects the repeller and attractor,
and the other generators form a robust blender, which can bring the points from a small neighborhood of the attractor to
an arbitrarily small neighborhood of the repeller. This is a model
setting for partially hyperbolic systems with one central direction.
We show that, under additional conditions on $\frac{f''}{f'}$ and the Schwarzian derivative, the above semigroups exhibit,
$C^r$-generically for any $r \geq 3$, arbitrarily fast growth of the number of periodic points as a function of the period.
We also show that a $C^r$-generic semigroup from the class under consideration supports an ultimately complicated behavior called universal dynamics.

\end{abstract}

{\footnotesize{2010 \emph{Mathematics Subject Classification}:
 37C35;
 37D30,
 37E05,
 37H05,
 37D45.}

}

\section{Introduction}\label{sec:intro}

One of the great mysteries of dynamical chaos is its extreme richness. Only in uniformly hyperbolic
systems the variability of chaotic dynamics can be controlled:
every basic set in Axiom A systems has a finite Markov partition \cite{Bow},
which implies that the chaotic behavior is self-similar in this case. However, once we leave uniformly hyperbolic systems,
the self-similarity gets broken. It is typical for non-hyperbolic chaotic systems that going to longer time scales and finer phase
space scales exhibits dynamics which is not present on the previous scales. Moreover, the diversity of dynamics that emerge in this process
can be unlimited. As typical manifestations of such complexity,
in this paper we investigate the phenomena of the
{\it fast (super-exponential) growth of the number of periodic points} and
the {\it universal dynamics} in a class of non-hyperbolic systems.

The super-exponential growth of periodic points does not seem to be natural. By
Artin and Mazur \cite{AM} a generic polynomial map has only exponential growth of the number of
periodic orbits as a function of period. The same holds true for hyperbolic systems: as we mentioned,
their dynamics can be described by means of a finite Markov partitions and, consequently, the number of periodic orbits can grow
at most exponentially. Notice also the result of
Martens, de Melo, and van Strien \cite{MMS} who implies
 theexponential growth of periodic points for $C^r$-endomorphisms
 ($r \geq 2$) of the unit interval with non-flat critical points.
 However, in other situations the super-exponential growth
 appears generically and is a characteristic feature of the wild behavior.
 The $C^r$-genericity of the super-exponential growth for systems
 in the Newhouse domain (the open region of diffeomorphisms
 with robust homoclinic tangencies \cite{Ne}) was discovered
 by Kaloshin \cite{K}. He noticed that the highly degenerate
 local bifurcations of the systems in the Newhouse domain,
 which were found in \cite{GTS92,GTS93},
 create large numbers of periodic orbits of the same period,
 and this, generically, leads to the super-exponential growth.
 By a similar strategy, the $C^1$-genericity of the super-exponential growth
 was shown in \cite{BDF} for systems with robust heterodimensional cycles.

Another important phenomenon observed in the non-hyperbolic setting is the generic occurrence of universal dynamics.
This phenomenon, discovered in \cite{GTS92,GTS93}, can be formally described
by the concept of a universal map which was introduced in \cite{BDwild,T03}.
A map is $(d,r)$-universal if the set of its iterations, each restricted to an appropriate region
in the phase space and written in certain rescaled coordinates, approximates all possible maps
from a unit ball to $\mathbb{R}^d$ with arbitrarily good precision
in the $C^r$-topology.

By definition, a single universal map mimics
all of the $d$-dimensional dynamics, i.e., it gives an example
of chaos of ultimate complexity.
The existence of universal maps may, at the first glance, seem improbable.
Indeed, if one considers one dimensional dynamics, it is an easy consequence of Belitsky-Mather theory \cite{Ma,Bel} that
there exists no $C^2$-universal diffeomorphisms of an interval, see \cite{Tu}.
However, when the dimension of the phase space is $2$ or higher,
under the presence of robust non-hyperbolicity,
universal maps do exist in abundance: they form a residual
subset of certain open regions in the space of dynamical systems.
Namely, $C^1$-universal maps are generic
in the Bonatti-D\'{i}az domain (the class of systems with
robustly non-dominated heterodimensional cycles, see
\cite{BDwild}), and $(d,r)$-universal
maps with arbitrarily large $r\geq 2$ are generic in the Newhouse domain
for $d=2$ \cite{T03,GTS7,Tu}.

One may wonder if there is another non-uniformly hyperbolic setting
which leads to the wild behavior as we saw above.
According to the famous conjecture of Palis,
the mechanism of non-hyperbolicity is ascribed to the existence of
homoclinic tangency or heterodimensional cycles
(see \cite{BDV, P00} for the precise statement and relevant discussions).
Thus it is interesting to ask if such degenerated behavior is observable
in systems exhibiting robust heterodimensional cycles.

By the pioneering work of Bonatti and D\'{i}az \cite{BDblender},
we certainly know that
the occurrence of heterodimensional cycles can be $C^1$-robust
(as a result, $C^r$-robust) through the presence of blenders (see \cite{BDV}).
In such situation, it is not hard to prove that $C^1$-generic dynamical
systems exhibits super-exponential growth of the number of
periodic points (see \cite{BDF}) and
has universal dynamics to the center direction \cite{BoBo}.
However, it is not obvious if such phenomena occur under
higher regularity settings.

In this article, aiming at the understanding of the behavior
in generic $C^r$-systems with robust heterodimensional cycles, where
$r \geq 2$,
we investigate {\it semigroup actions on the interval}.
As suggested in several papers (see for example
\cite{DGR,BB,GI}),
they serve as simplified models of the
systems with robust heterodimensional cycles and, more generally,
partially hyperbolic systems with one central direction.
We prove that under certain mild,
non-hyperbolicity conditions,
such systems do exhibit the wild behavior.

Let us briefly see the statement of our main result;
for basic definitions and precise statements, see Section~\ref{sec:mainre}.
We investigate semigroup actions on the interval $I=[0, 1]$
generated by three maps $f_i:I \to I$ $(i=0, 1, 2)$, which are
smooth and strictly increasing. We assume that $f_0$ contains a repeller-attractor heteroclinic connection.
We also assume that $(f_1, f_2)$ is a persistent blender
on an interval containing the repeller-attractor heteroclinic (see Section~\ref{sec:mainre} for the definitions).
The fact that $(f_1, f_2)$ is a blender means that the action of $(f_1,f_2)$ spreads points over the interval.
This implies that there are orbits starting near the attracting fixed point of $f_0$ and
ending near the repelling fixed point.  Thus there are transient
orbits going back and forward between the repelling point and the attracting
point, which supports the non-hyperbolic behavior.

Let us introduce the $C^r$-topology in the space of such semigroups.
Then, our result can be stated as follows:

\medskip

{\bf Main Result}:
{\it For each $r \geq 1$, there exists a non-empty, $C^r$-open region
$\cW^r$ in the space of semigroups satisfying above conditions,
in which $C^r$-generic elements exhibit
super-exponential growth of periodic points
(Theorem~\ref{thm:arbitrary growth}) and
have $C^r$-universal dynamics (Theorem~\ref{thm:universal}).}

\medskip

We also prove that the itineraries along which we observe super-exponential growth
are quite abundant (see Theorem~\ref{thm:growth along path}).

Local genericity of super-exponential growth of periodic points for $r=1$
is just an easy analog of a result in \cite{BDF}.
Our main contribution is the case $r \geq 2$.
One interesting feature of our result is that the dynamics we describe is controlled by derivatives of order higher than $1$.
Usually, the conditions on the generic dynamics are formulated in terms of
the first derivative only. It seems probable
that no conditions involving higher order derivatives restrict
the possible richness of chaotic dynamics in systems with homoclinic tangencies \cite{T96,GST8,T10,Tu,Gour}.
Surprisingly, as the results of the present paper suggest,
this ``the first derivative alone'' principle should
not be applicable to the description of dynamics near
semigroup actions and, accordingly, robust heterodimensional cycles.

The class $\cW^r$ is described 
in terms of the second order derivative if $r \geq 2$
and the third order derivative if $r \geq 3$ (see Section~{\ref{sec:T}}).
A simple speculation reveals the necessity of such conditions.
Consider the semigroups for which all the maps $f_i$ have strictly positive
second derivatives everywhere. Any composition of increasing convex functions is, obviously, convex.
Therefore, any composition of the maps $f_i$ cannot have more than 2 fixed points, i.e., it cannot approximate
a map with a higher number of fixed points, so any such semigroup can not have universal dynamics.
Neither can it have a super-exponential growth: a periodic orbit corresponds to a periodic itinerary;
if for for each such itinerary the corresponding period map cannot have more than 2 fixed points, then the number
of periodic orbits is not more than twice the number of the periodic itineraries, and the latter grows exponentially with period.
The same holds true if all the maps are concave. Similarly, if the Schwarzian derivative is negative for all the maps $f_i$ (or positive for all of 
them), then every composition of $f_i$ has negative (resp. positive) Schwarzian derivative too. This restricts
the number of periodic points for every given itinerary by $3$, so we do need some condition on the
third derivative, involving the Schwarzian derivative, in order to have universal dynamics and/or the superexponential growth.

As we have already mentioned, the semigroups which
we study here can serve as simplified models for the study
of systems with heterodimensional cycles in partially hyperbolic systems.
Therefore, our result suggests that the rate of the growth of the number
of periodic orbits for a $C^r$-generic system
having robust heterodimensional cycles
can be determined by different factors for
$r=1$, $r=2$ and $r\geq 3$.


Let us briefly explain the scheme of the proof of the Theorem.
First, we construct an $r$-flat
periodic point by an arbitrarily small perturbation of the maps
(see Section~\ref{sec:flat}).
A periodic point is $r$-flat
if it is neutral (i.e., the first derivative of the period map at
that point is equal to $1$) and the derivatives of orders from $2$ to $r$ vanish
at this point, i.e., the period map is given by $x\mapsto x+o(x^r)$.

The construction of such periodic points is done by induction in the order of flatness:
we show that if the semigroup with the persistent blender has a
sufficiently large number of $k$-flat periodic points with $k \geq 3$,
then a $(k+1)$-periodic point can be created by an arbitrarily small
$C^r$-perturbation of the system. Moreover, the existence of a persistent blender
allows to place this point within any given interval and also let its itinerary to follow
the itinerary of any given orbit as long as we want. The perturbation can be localized
in an arbitrarily small neighborhood of some finite set of points, so the process
can be repeated without destroying any finite number of the flat points created
at the previous steps.

A $k$-flat periodic point corresponds to a codimension-$k$ bifurcation.
A local unfolding of this bifurcation can only decrease the codimension.
However, the presence of homoclinic and heteroclinic points due to the blender
allows for creating bifurcations of codimension $(k+1)$. The fact that
homoclinic bifurcations can lead to an unbounded increase in the codimension
of accompanying bifurcations of periodic orbits was discovered in \cite{GTS92,GTS93}.
It was related to the presence of hidden bifurcation parameters
(moduli of conjugacy) at typical homoclinic bifurcations \cite{GTSt}.
The strategy of our proof here has the same flavor as the proof of a similar
result for systems with homoclinic tangencies \cite{GTS1,GTS7}.
Indeed, the main argument is based on the calculation of a superposition of
polynomial maps. However, the actual construction is quite different:
in the case of a homoclinic tangency the argument unfolds in a neighborhood
of a critical point, while the maps we consider here
are diffeomorphisms of an interval, which results in
a different algebraic structure.

As we have seen before,
the information on the signature of the second and Schwarzian
derivatives is important in our case.
For example, the induction cannot start at $k\leq 2$.
If we write the period map near a $k$-flat point as
$x\mapsto x+ a_{k+1} x^{k+1} + o(x^{k+1})$,
then the sign of $a_2$ for a $1$-flat point
is the sign of the second derivative;
the sign of $a_3$ for a $2$-flat point is the sign of the
Schwarzian derivative. Therefore, for a semigroup
whose generators have
second (resp. Schwarzian) derivatives with definite signature,
we see that we cannot produce 2- (resp. 3-) flat points.
In our proof, $2$-flat periodic points are created by a perturbation
over $1$-flat periodic points with different signs
of $a_2$. Such $1$-flat points are obtained by a perturbation of
a heteroclinic cycle which includes a pair of repeller-attractor
heteroclinic connections of different characteristics in terms of the second derivative; the existence of such heteroclinics is a part of 
the conditions that describe the set $\cW^2$.
Similarly, when the same heteroclinics carry also an opposite sign of
Schwarzian derivative (this condition will is assumed in the definition of $\cW^r$ with $r\geq 3$),
we are able to make the resulting $2$-flat points having opposite signs of $a_3$.
Then, a perturbation including such $2$-flat points creates a $3$-flat point.
After that the above described induction can start, as the sign of $a_{k+1}$ does not play a role if $k\geq 3$.
The difference between the case $k \leq 3$ and $k >3$ will be elucidated
through the proof of local algebraic lemmas given in Section~\ref{sec:cancel}.

Once the possibility to create an $r$-flat periodic point by an arbitrarily
small perturbation is established, the proof of the generic super-exponential growth
of periodic points is done by the Kaloshin argument (see Section~\ref{sec:flat}). The creation 
of universal dynamics out of the abundant $r$-flat points is less straightforward. The
idea of creating universal dynamics by perturbing
a flat periodic point can be traced back to Ruelle and Takens work \cite{RT}.
This task is not trivial in the $C^r$-topology setting with large $r$. While
it was solved in \cite{Tu} for maps of dimension $2$ and higher, the methods of \cite{Tu} are inapplicable in the one-dimensional setting.
Therefore, we derive the genericity of the $C^r$-universal maps in ${\cW}^r$ from the occurrence of
$r$-flat points by employing a completely different technique; we also make a substantial use
of the existence of a blender which gives us the freedom in choosing orbit itineraries (see Section~\ref{sec:universal}).

Let us discuss one technical matter. Even though we use localized (hence, non-analytic)
perturbations in our proof, it seems that the construction can be
modified (in the spirit of \cite{BT,GTS7}) in order to encompass
the analytic case, and we believe that semigroups that have $C^\omega$-universal semigroups
are generic in $\cW^\omega$.  However, we do not know whether
the generic super-exponential growth holds in the analytic case or not.
The class of semigroups we consider in this paper can, possibly, serve as the simplest non-trivial
class of systems for which this question can be investigated.

In conclusion, we also remark that the analysis of
semigroup actions presented in this paper has similar flavor
with the study of dynamics of cocycles. The reader will find similarities of
arguments among papers such as \cite{BDP,ABY,NP}.
Indeed, these two objects are tightly related: given a semigroup action, one
can construct a diffeomorphism cocycle
dynamics over shift spaces by taking a skew product.
Consequently, all the results we obtain immediately give us corresponding results for the
skew-product systems.

We also expect that the statement similar to our
theorems should hold true for partially hyperbolic diffeomorphisms
with heterodimensional cycles with one-dimensional central direction.
However, compared to the semigroups case, the holonomies
along the center foliation of a generic partially hyperbolic map have low
regularity, which prevents a direct transfer of the results
we obtain for one-dimensional semigroups
to multi-dimensional partially hyperbolic diffeomorphisms.
Meanwhile, we believe that the main techniques are transferrable,
and should be useful for the further work in this direction.

In the next section,
we start rigorous arguments:
we give basic definitions and precise statement
of our results. The organization of this paper is explained at the
end of Section~\ref{sec:mainre}.

\medskip
{\bf Acknowledgement}
KS and DT are grateful for the hospitality of Department of Mathematics of Kyoto University.
This paper is supported by GCOE program of Kyoto University, JSPS KAKENHI Grant-in-Aid for Young Scientists (A) (22684003),
Scientific Research (C) (26400085), and JSPS Fellows (26$\cdot$1121), by Grant No. 14-41-00044 of RSF (Russia), and by 
the Royal Society grant IE141468.

\section{Main results}\label{sec:mainre}
As explained in Section~\ref{sec:intro}, the aim of this paper is to show that
generic semigroup actions in a certain open subset of the space of actions has a wild behavior.
The first theorem (Theorem \ref{thm:arbitrary growth}) asserts that an arbitrarily fast growth of the number of attracting periodic points
is generic in $\cW^r$, an open region whose precise definition is explained below.
The second theorem (Theorem \ref{thm:universal}) asserts that a generic triple in $\cW^r$ generates `'universal dynamics''.
The last theorem (Theorem \ref{thm:growth along path}) asserts that under
certain additional conditions, the number of attracting periodic points grows, along a generic infinite word,
faster than any given function of the period for a generic triple in $\cW^r$. 

\subsection{Space of semigroup actions and its open subset $\cW^r$}
\label{sec:T}
In this subsection, we prepare basic terminologies used throughout the paper.

For a finite set $S$, we denote the set $\bigsqcup_{n \geq 0} S^n$
by $S^*$.
The set $S^*$ is called the set of {\it words} of alphabets $S$.
Under concatenation of words,
$S^*$ forms a semigroup.
For $1 \leq r \leq \infty$,
 let $\cE^r$ be the space of orientation-preserving $C^r$ embeddings
 from $[0,1]$ to $(0,1)$ endowed with $C^r$-topology.
By composition of maps, it is also a semigroup.
We write $\cA^r(S)$ for the set of families $(f_s)_{s \in S}$
 of maps in $\cE^r$ indexed by $S$.
This set is endowed with the product topology
 of the $C^r$-topology of $\cE^r$.
For a family $\rho=(f_s)_{s \in S} \in \cA^r(S)$
 and a word $\omega=s_n\dots s_1 \in S^*$,
 we define a map $\rho^\omega$ in $\cE^r$ by
 $\rho^\omega=f_{s_n} \circ \dots \circ f_{s_1}$.
Then, the map $\omega \mapsto \rho^\omega$ is a homomorphism
 between $S^*$ and $\cE^r$.
This homomorphism is called the {\it semigroup action} generated by
 $\rho=(f_s)_{s \in S}$.
It is easy to see that any homomorphism from $S^*$ to $\cE^r$
 is generated by a family in $\cA^r(S)$.

For $\rho \in \cA^r(S)$ and $x \in [0,1]$,
 let  $\cO_+(x,\rho)$ be the {\it forward orbit of $x$ under $\rho$}
 $\{\rho^\omega(p) \mid \omega \in S^*\}$.
We call an element $\rho \in \cA^r(S)$ a {\it blender}
 on a closed interval $J \subset (0,1)$
if the closure of $\cO_+(x,\rho)$ contains $J$ for any $x \in J$.
We say that a blender $\rho$ on $J$ is {\it $C^r$-persistent}
if any semigroup action which is $C^r$-close to $\rho$ is a blender on $J$.
It is known that a $C^r$-persistent blender exists.
For example, suppose that $(f_1,f_2) \in \cA^r(\{1,2\})$ satisfies
 that $f_1'<1$ and $f_2'<1$ on an closed interval $[a,b] \subset (0,1)$,
 $f_1(a)=a$, $f_2(b)=b$, and $f_1(b)>f_2(a)$.
 Then $(f_1,f_2)$ is a $C^r$-persistent blender for any closed interval
 $J \subset (a,b)$ (see \cite[Example 1]{Sh}).

Let $f$ be a map in $\cE^r$ with $1 \leq r \leq \infty$.
A pair $(p,q)$ of points in $[0,1]$ is a {\it repeller-attractor pair}
 if $p$ and $q$ are fixed points of $f$
 such that $f'(p)>1>f'(q)$ and $W^u(p) \cap W^s(q) \neq \emptyset$.
A point in $W^u(p) \cap W^s(q)$ is called a {\it heteroclinic point}
 of $(p,q)$.
We define two quantities $\tau_A(z_0,f), \tau_S(z_0,f) \in \{\pm 1,0\}$
 for a heteroclinic point $z_0$ as follows.
For a map $g \in \cE^r$ and $x \in [0,1]$, let
 $A(g)_x$ and $S(g)_x$  be the {\it non-linearity}
 and the {\it Schwarzian derivative} of $g$ at $x$,
 defined as follows:
\begin{align*}
 A(g)_x & = \frac{g''(x)}{g'(x)},&
 S(g)_x & = \frac{g'''(x)}{g'(x)} -\frac32\left(\frac{g''(x)}{g'(x)}\right)^2,
\end{align*}
 where $A(g)_x$ is defined only if $r \geq 2$
 and $S(g)_x$ is defined only if $r \geq 3$.
When $r \geq 2$, there exist normalized $C^r$-linearizations of $f$
 $\vphi:W^u(p) \ra \RR$ and $\psi:W^s(q) \ra \RR$ at $p$ and $q$, i.e.,
 orientation preserving diffeomorphisms satisfying
 $\vphi \circ f(x)=\lambda_p \vphi(x)$,
 $\psi \circ f(x)=\lambda_q \psi(x)$,
 and $\vphi'(p)=\phi'(q)=1$, where $\lambda_p=f'(p)$ and $\lambda_q=f'(q)$.
For a heteroclinic point $z_0 \in W^u(p) \cap W^s(q)$, set
\begin{align*}
 \tau_A(z_0,f) & =\sgn(A(\psi \circ \vphi^{-1})_{\vphi(z_0)}),\\
 \tau_S(z_0,f) & =\sgn(S(\psi \circ \vphi^{-1})_{\vphi(z_0)}),
\end{align*}
 where $\sgn:\RR \ra \{0, \pm 1\}$ is the sign function.
We call the pair $(\tau_A(z_0,f),\tau_S(z_0,f))$ {\it the sign} of $z_0$
if $f$ is of class $C^3$.
If $f$ is only $C^2$, then the sign of $z_0$ is just one number,
 $\tau_A(z_0,f)$.

The normalized $C^r$-linearizations $\vphi$ and $\psi$ are known to
exist uniquely if $r \geq 2$ (see \cite[Theorem 2]{S}). Therefore, the map $\psi \circ \vphi^{-1}$
is uniquely defined and is an invariant (a functional modulus) of the smooth conjugacy of maps of the interval \cite{Ma, Bel}.
Hence, $\tau_A(z_0,f)$ is well-defined if $r \geq 2$, and $\tau_S(z_0,f)$ is well-defined if $r \geq 3$. Moreover,
the sign of the heteroclinic point is invariant with respect to $C^r$-smooth coordinate transformations, and it is the same
for every point of the orbit of $z_0$ by $f$.

Let $\cW^1$ be the set of $\rho=(f_0,f_1,f_2) \in \cA^1(\{0,1,2\})$
 which satisfy the following conditions:
\begin{description}
 \item[Existence of Blender]
 $(f_1,f_2) \in \cA^1(\{1,2\})$ is a $C^1$-persistent blender on a closed interval $J \subset [0,1]$.
 \item[Non-hyperbolicity]
 $f_0$ admits a repeller-attractor pair $(p,q)$ in $\Int J$.
\end{description}
The set $\cW^1$ is an open subset of $\cA^1(\{0,1,2\})$.
Let $\cW^2$ be the set of $\rho=(f_0,f_1,f_2) \in \cW^1 \cap \cA^2(\{0,1,2\})$
 which satisfy the following condition:
\begin{description}
 \item[Sign condition I]
 $f_0$ admits repeller-attractor pairs $(p_1,q_1)$ and $(p_2,q_2)$ in $\Int J$
 (the case where $(p_1,q_1)=(p_2,q_2)$ is allowed)
 and there exist heteroclinic points $z_1 \in W^u(p_1) \cap W^s(q_1)$
  and $z_2 \in W^u(p_2) \cap W^s(q_2)$ for $f_0$
 such that $\tau_A(z_1,f_0) \cdot \tau_A(z_2,f_0)<0$.
\end{description}
Finally,
let $\cW^3$ be the set of $\rho=(f_0,f_1,f_2) \in \cW^2 \cap \cA^3(\{0,1,2\})$
which satisfy the following condition:
\begin{description}
 \item[Sign condition II]
 $f_0$ admits repeller-attractor pairs $(p_3,q_3)$ and $(p_4,q_4)$
 in $\Int J$ (the case where $(p_i,q_i)=(p_j,q_j)$
 for some $1 \leq i< j \leq 4$ is allowed)
 and there exist heteroclinic points $z_3 \in W^u(p_3) \cap W^s(q_3)$
  and $z_4 \in W^u(p_4) \cap W^s(q_4)$ for $f_0$
 such that $\tau_S(z_3,f_0) \cdot \tau_S(z_4,f_0)<0$.
\end{description}
Remark that $\cW^3 \subset \cW^2 \subset \cW^1$.
 We also define $\cW^r=\cW^3\cap \cA^r(\{0,1,2\})$ for $r\geq 4$
(i.e. we do not apply any other conditions except for smoothness if $r\geq 4$).
As we will see in Proposition {\ref{prop:sign}},
 if $\tau_A(z_0,f_0)\neq 0$,
 then it does not change at $C^2$-small perturbations,
 and if $\tau_S(z_0,f_0)\neq 0$,
 then it does not change at $C^3$-small perturbations.
Hence, $\cW^r$ is an open subset of $\cA^r(\{0,1,2\})$ for any $r \geq 1$.
In Section 9 we give simple sufficient criteria for the fulfillment
 of the sign conditions,
 which do not require the computation of
 the Belitsky-Mather invariant $\psi \circ \varphi^{-1}$.
These criteria are formulated in terms of the first derivatives only,
 so we can conclude that $\cW^1$ has a $C^1$ open subset
 where each $C^r$ semigroup belongs to $W^r$.

\subsection{Arbitrary growth of the number of periodic points}
\label{sec:result growth}
For a map $f \in \cE^r$, set
\[
 \Fix(f)  =\{x \in [0,1] \mid f(x)=x\}, \quad
 \Fix_a(f)   =\{x \in \Fix(f) \mid f'(x)<1\}.
 \]
The following is our first result.
\begin{thm}
\label{thm:arbitrary growth}
For any $1 \leq r \leq \infty$
and any sequence $\ub{a}=(a_n)_{n=1}^\infty$ of positive integers,
a generic (in the sense of Baire) element $\rho$ in $\cW^r$ satisfies 
\begin{equation*}
 \limsup_{n \ra \infty}
 \frac{\sum_{\omega \in \{1,2,3\}^n}\# \Fix_a(\rho^\omega)}{a_n}=\infty.
\end{equation*}
\end{thm}

Thus, in $\cW^r$, semigroups which exhibit
arbitrarily fast growth of $\# \Fix_a(\rho^\omega)$ are quite abundant.
Meanwhile, notice that every semigroup action can be $C^r$-approximated
by the one generated by polynomial maps, and for these maps, by estimating the growth of
the degree, we can easily see that $\# \Fix_a(\rho^\omega)$ grows
at most at an exponential rate
(this is analogous to the theorem by Artin and Mazur \cite{AM}).


Theorem \ref{thm:arbitrary growth} shows an interesting contrast between
the average growth of the number of periodic points and the growth along almost every infinite word.
Let $\mu$ be the uniform distribution on $\{0,1,2\}$.
We denote the product probability on $\{0,1,2\}^\infty$ by $\mu^\infty$.
For $\ub{\omega}=\cdots s_2 s_1 \in \{0,1,2\}^\infty$
and $n \geq 1$, set $\ub{\omega}|_n=s_n\dots s_1$.
As we will see in Section \ref{sec:criterion}, the set
\begin{equation*}
 \cW^r_{\mathrm{att}}=\{\rho=(f_0,f_1,f_2) \in \cW^r \mid
 |f'_0|_\infty \cdot |f'_1|_\infty\cdot |f'_2|_\infty<1\}
\end{equation*}
 is non-empty, where $|h|_\infty=\sup_{x \in [0,1]}|h(x)|$.
By the law of large numbers,
\begin{equation*}
\lim_{n \ra \infty}\frac{\#\{m \in \{1,\dots,n\} \mid s_m=s\}}{n}
 =\frac{1}{3}
\end{equation*}
 for any $s \in \{0,1,2\}$
 and $\mu_\infty$-almost every $\ub{\omega} \in \{0,1,2\}^\infty$.
By using this, we can see that
\begin{equation*}
 \limsup_{n \ra \infty}|(\rho^{\ub{\omega}|_n})'|_{\infty}^{1/n}
 \leq (|f'_0|_\infty \cdot |f'_1|_\infty\cdot |f'_2|_\infty)^{1/3}<1
\end{equation*}
 holds for any $\rho \in \cW^r_{\mathrm{att}}$
 and $\mu_\infty$-almost every infinite word
 $\ub{\omega}\in \{0,1,2\}^\infty$.
Hence,
\begin{equation*}
 \lim_{n \ra \infty} \#\Fix(\rho^{\ub{\omega}|_n})=1
 \hspace{3mm} \text{for }\mu^\infty\text {-almost every } \ub{\omega}.
\end{equation*}
On the other hand,  Theorem \ref{thm:arbitrary growth}
 implies that semigroup actions $\rho$ satisfying
\begin{align*}
\lefteqn{ \limsup_{n \ra \infty}\frac{1}{a_n}\int_{\{0,1,2\}^\infty}
 \#\Fix_a(\rho^{\ub{\omega}|_n})
 d\mu^\infty(\ub{\omega})}\\
 & = \limsup_{n \ra \infty}\frac{1}{a_n}
\frac{\sum_{\omega \in \{0,1,2\}^n}\#\Fix_a(\rho^\omega)}{3^n}
 =\infty
\end{align*}
 is generic in $\cW^r_{\mathrm{att}}$
 for any sequence $\ub{a}=(a_n)_{n=1}^\infty$
 of positive integers (by applying Theorem~\ref{thm:arbitrary growth}
 replacing $(a_n)$ by $(3^na_n)$).
Therefore, the $\mu^\infty$-averaged growth
 of $\#\Fix_a(\rho^\omega)$
 and the growth of $\#\Fix_a(\rho^{\ub{\omega}|_n})$
 along $\mu^\infty$-almost every infinite words $\ub{\omega}$
 are completely different for generic $\rho$ in $\cW^r_{\mathrm{att}}$.

\subsection{Universal dynamics}
\label{sec:result universal}
For finite sets $S$ and $S'$, and
 families $\rho=(f_s)_{s \in S} \in \cA^r(S)$
 and $\rho'=(g_{s'})_{s' \in S'} \in \cA^r(S')$,
 we say that $\rho$ {\it realizes} $\rho'$
 if there exists a closed interval $I \subset [0,1]$,
 a diffeomorphism $\Phi:[0,1] \ra I$, and
 a family $(\omega_{s'})_{s' \in S'}$ of words in $S^*$ such that
 $g_{s'}=\Phi^{-1} \circ (\rho^{\omega_{s'}}|_I) \circ \Phi$ holds
 for any $s' \in S'$.
In other words,
 the semigroup actions generated by $\rho'=(g_{s'})_{s' \in S'}$
 and $(\rho^{\omega_{s'}}|_I)_{s' \in S'}$
 are conjugate by the diffeomorphism $\Phi$.
We say that $\rho \in \cA^r(S)$ generates
 a {\it universal semigroup}
 if for each finite set $S'$ there exists a dense subset ${\mathcal D}_{S'}$
 of $\cA^r(S')$ such that $\rho$ realizes any $\rho' \in {\mathcal D}_{S'}$.

The following is our second result.
\begin{thm}
\label{thm:universal}
For any $1 \leq r \leq \infty$,
a generic element in $\cW^r$ generates a universal semigroup.
\end{thm}

\subsection{Wild behavior along generic infinite words}
Under an additional mild condition,
 the semigroup generated by a generic element of $\cW^r$ exhibits
 wild behavior along generic infinite words.
Let $\cW^r_\#$ be the set consisting of elements
 $\rho=(f_0,f_1,f_2)$ of $\cW^r$
 such that $\cO_+(x,\rho) \cap \Int J \neq \emptyset$
 for each $x \in [0,1]$, where $J$ is the interval on which
 $(f_1,f_2)$ is a persistent blender.
The set $\cW^r_\#$ is a non-empty open subset of $\cA^r(\{0,1,2\})$.
We furnish the product topology on $\{0,1,2\}^\infty$
induced by the discrete topology of the set $\{0,1,2\}$.
\begin{thm}
\label{thm:growth along path}
For any sequence $(a_n)_{n=1}^\infty$ of positive integers,
a generic $\rho \in \cW^r_\#$ satisfies
\begin{equation*}
 \limsup_{n \ra \infty}\frac{\#\Fix_a(\rho^{\ub{\omega}|_n})}{a_n}=\infty
\end{equation*}
for every generic infinite word $\ub{\omega} \in \{0,1,2\}^\infty$.
\end{thm}
As we will see in Section {\ref{sec:criterion}},
 $\cW^r_\# \cap \cW^r_{\mathrm{att}}$ is non-empty.
For any $\rho \in \cW^r_\# \cap \cW^r_{\mathrm{att}}$
 and $\mu^\infty$-almost every $\ub{\omega}$,
 $\rho^{\ub{\omega}|_n}$ is a uniform contraction
 for any sufficiently large $n$.
This implies that the generic infinite words
in Theorem~\ref{thm:growth along path}
form a null subset of $\{0,1,2\}^\infty$ with respect to
the probability measure $\mu^\infty$.

\subsection{Organization of this paper}
The rest of this paper is organized as follows.
In Section~\ref{sec:cancel}, we prepare several local algebraic results
 about the composition of germs. In Section~\ref{sec:connecting},
 we prepare the notation for the perturbation of semigroups and
 give several lemmas, which produce orbits that realize desired germs.
In Section~\ref{sec:flat}, by using the techniques which we prepare
 in Sections~\ref{sec:cancel} and \ref{sec:connecting},
 we give the induction argument
 producing $r$-flat periodic orbits, and complete the proof
 of Theorem~\ref{thm:arbitrary growth}.
In Section~\ref{sec:universal},
 we prove Theorem~\ref{thm:universal}
 by using the construction of $r$-flat
 periodic orbits (which is already obtained in Section~\ref{sec:flat})
 together with a lemma about the decomposition of diffeomorphisms
 on the interval (Lemma~\ref{lemma:composition}).
In Section~\ref{sec:infinite path},
 we prove Theorem~\ref{thm:growth along path}.
The proof is done by a careful reiteration of the proof of
 Theorem~\ref{thm:arbitrary growth}
 together with a genericity argument (Lemma~\ref{lemma:generic path}).
In Section~\ref{sec:sign condition}, we prove that the sign
 condition which defines the set $\mathcal{W}^r$ is $C^r$-open.
Finally, in Section~\ref{sec:criterion},
 we give a simple sufficient condition for
 the fulfillment of the sign conditions.
As an application, we give a simple polynomial example
 for a semigroup in $\cW^\infty_\# \cap \cW^\infty_{\mathrm{att}}$.


\section{Cancellation of germs}\label{sec:cancel}

As explained in the introduction, for the proof of Theorem~\ref{thm:arbitrary growth} and Theorem~\ref{thm:universal},
we first produce $r$-flat periodic points by an arbitrarily small perturbation.
The construction of such periodic points will be done inductively.
In this section, we derive local algebraic propositions needed
 for the inductive step.
Namely, we show how to obtain an $(r+1)$-flat germ
 as a composition of iterations of $r$-flat germs.

Let $\cD^r$ be the set of germs
 of an orientation-preserving local $C^r$-diffeomorphisms of $\RR$
 with a fixed point at the origin.
We simply write $\cD$ for $\cD^{\infty}$.
For $F \in \cD$ and $s \geq 1$, we denote by $F^{(s)}$ (or $F'$ for $s=1$)
 the $s$-th derivative of $F$ at $0$.
We define a pseudo-distance $d$ on $\cD$ by
\begin{equation*}
d(F,G)=\sum_{r=1}^\infty 2^{-r}
 \left(\frac{|F^{(r)}-G^{(r)}|}{1+|F^{(r)}-G^{(r)}|}\right).
\end{equation*}
The pseudo-distance $d$ defines a topology on $\cD$.
This topology is non-Hausdorff.
Indeed, a germ $F$ with $F'=1$ and $F^{(r)}=0$ for all $r \geq 2$
 is not separated from the identity germ $I$.
We say that $F \in \cD$ is \emph{$r$-flat}
 if $F'=1$ and $F^{(s)}=0$ for all $s=2,\ldots,r$.
The term $\infty$-flat will mean $r$-flat for every $r \geq 1$.
For $F \in \cD$, let $A(F)$ and $S(F)$ be the non-linearity
 and the Schiwarzian derivative of $F$ at $0$ respectively.
The {\it sign} of a germ $F \in \cD^r$ is the pair $(\sgn(A(F)), \sgn(S(F)))$
We say that two germs $F,G \in \cD^r$  have the same or opposite signs
 if both $A(F) \cdot A(G)$ and $S(F) \cdot S(G)$
 are positive or, resp., negative.
For $f \in \Diff^r([0,1])$ and $x \in (0,1)$,
 we define a germ $[f]_x$ in $\cD^r$ by $[f]_x(y)=f(x+y)-f(x)$.

We start with recalling the fact (Lemma \ref{lemma:exp}) that
 any germ in $\cD$ equals to a time-one map of a local flow up to order $r$,
 for each fixed $r \geq 1$.
For $\alpha>0$, let $L_\alpha$ be the element of $\cD$ given by $L_\alpha(x)=\alpha x$.
\begin{lemma}
\label{lemma:exp}
For any $F \in \cD$ and $r \in [1, +\infty)$,
 there exists a continuous family of germs $(F^t)_{t \in \RR}$ in $\cD$
 such that $F^0$ is the identity map,
 $F^1=F+o(x^r)$, and $F^t \circ F^{t'}=F^{t+t'}+o(x^r)$
 for any $t,t' \in \RR$.
\end{lemma}
\begin{proof}
Recall that $F$ is orientation preserving. Put $\alpha=F'>0$.
If $\alpha \neq 1$, then $F$ is smoothly linearizable at $0$. This means
that there exists $\Phi \in \cD$ such that
 $\Phi \circ F \circ \Phi^{-1}(x)=\alpha x$.
In this case, the family
 $(\Phi^{-1} \circ L_{\alpha^t} \circ \Phi)_{t \in \RR}$ satisfies
 the required properties.

By $\cD(r)$, we denote the subgroup of $\cD$ consisting of $r$-flat elements,
 where the group operation is given by the composition of germs.
Suppose that $F'=1$. Then $F$ belongs to $\cD(1)$.
The group $\cD(1)/\cD(r)$ is a finite-dimensional connected
 nilpotent Lie group and it is well-known that
 the exponential map is surjective for such Lie groups.
Hence, there exists an element $\xi$ in the Lie algebra of $\cD(1)/\cD(r)$
 such that $\exp(\xi)=F+o(x^r)$.
Then the family $(\exp(t\xi))_{t \in \RR}$ satisfies the required properties.
\end{proof}

\begin{rmk}
There is an explicit inductive construction of the family
$(F_t)_{t \in \RR}$ for the case $F'=1$:
The constant family $(F_1^t\equiv I)_{t \in \RR}$ satisfies
the required condition for $r=1$.
Suppose that we have a family $(F_r^t)_{t \in \RR}$ which satisfies
the required condition for some $r \geq 1$.
Put $a=\{F^{(r+1)}(0)-F_r^{(r+1)}(0)\}/(r+1)!$
and let $(G^t)_{t \in \RR}$ be the germ of local flow generated
by the vector field $ax^{r+1}(\del/\del x)$ at $0$.
We set $F_{r+1}^t=F_r^t \circ G^t$.
Then, $F_{r+1}^1(x)=F_r^1(x)+ax^{r+1}+o(x^{r+1})=F(x)+o(x^{r+1})$.
Since the germ $G^t$ is $r$-flat,
 the maps $F_r^t$ and $G^t$ commutes up to order $x^{r+1}$
(see also Remark~\ref{rmk.tech}).
Hence, we have $F_{r+1}^{t+t'}=F_{r+1}^t \circ F_{r+1}^{t'}+o(x^{r+1})$.
Thus the family $(F_{r+1}^t)_{t \in \RR}$ satisfies the required condition
for $r+1$.
\end{rmk}

For $h \in \Diff^r([0,1])$ or $\Diff^r(\mathbb{R})$, the {\it support} of $h$,
denoted by $\mathrm{supp}\,h$, is the closure of $\{x \in [0,1] \mid h(x) \neq x\}$.
For a family of diffeomorphism $(h_t)$, its \emph{support} means the
closure of the union $\cup_t \mathrm{supp} \, h_t$.
%
%
The following lemma plays a key role in our construction of universal semigroups.
\begin{lemma}
\label{lemma:composition}
Let $r \in [2, +\infty]$. Let $I$ be a compact interval in $\RR$
and $F$ be an orientation-preserving $C^r$-diffeomorphism of $\RR$
such that $\supp(F) \subset I$.
Then, there exist one-parameter groups $(G^t)_{t \in \RR}$
and $(H^t)_{t \in \RR}$ of $C^r$-diffeomorphisms of $\RR$
and a compact interval $I'$ such that $F=G^1 \circ H^1$ on $I$
and the support of $(G^t)$ and $(H^t)$ are both contained in $I'$.
\end{lemma}
\begin{proof}
Take $0<\lambda<1$ such that
the map $F_\lambda(x)=\lambda F(x)$ is a uniform contraction on $\RR$. The contraction property
implies that the map $F_\lambda$ has a unique fixed point $p_*$. This fixed point is exponentially stable.
It follows that $F_\lambda$ is $C^r$-linearizable on $\RR$.
More precisely, there exists a $C^r$-diffeomorphism $\varphi$ of $\RR$
and a constant $\mu>0$ such that $F_\lambda \circ \varphi(x)=\varphi(\mu x)$ for any $x \in \RR$.

Put $\bar{G}^t(x)=\lambda^{-t}x$  and $\bar{H}^t(x)=\varphi(\mu^t \cdot \varphi^{-1}(x))$
for $x \in \RR$. These are one-parameter groups of diffeomorphisms of $\RR$.
Take a compact interval $I'$ whose interior contains $\bar{G}^t \circ F_\lambda(I)$
and $\bar{H}^t(I)$ for any $t \in [0,1]$.
By cutting-off the vector fields
 generating $(\bar{G}^t)$ and $(\bar{H}^t)$ outside $I'$,
 we obtain one-parameter groups $(G^t)_{t \in \RR}$
 and $(H^t)_{t \in \RR}$ of diffeomorphisms of $\RR$
 such that $G^t(x)=\bar{G}^t(x)$ and $H^t(x)=\bar{H}^t(x)$
 for any $x \in I \cup F_\lambda(I)$ and $t \in [0,1]$,
 and the supports of $G^t$ and $H^t$ are both contained in $I'$.
By construction, $G^1 \circ H^1(x)=\lambda^{-1} \cdot F_\lambda(x)=F(x)$
 for any $x \in I$.
\end{proof}

In the following, we give several lemmas on the cancellation of germs.
Their proofs will be done by calculating compositions of polynomials.
In the proofs, we will exploit the following elementary observations.
\begin{rmk}\label{rmk.tech}
 1. Non-linearities and Schwarzian derivatives satisfy
 the following \it{cocycle properties}:
 for $F, G \in \mathcal{D}$, we have
\begin{equation*}
 A(F \circ G)=A(F) \cdot (G')+ A(G), \quad
 S(F \circ G)=S(F) \cdot (G')^2+ S(G).
\end{equation*}
In particular, if the germs $F$ and $G$ are $1$-flat, then
\begin{equation*}
 A(F \circ G)=A(F)+ A(G), \quad
 S(F \circ G)=S(F)+ S(G).
\end{equation*}
2. Suppose that $F \in \cD$ is $1$-flat and $G \in \cD$ satisfies
 $G(x)=x+cx^{r+1}+o(x^{r+1})$ ({\it i.e}, $G$ is $r$-flat).
Then an easy computation shows that
\begin{equation*}
F \circ G(x) \equiv G \circ F(x) \equiv F(x)+cx^{r+1}
 \mod o(x^r).
\end{equation*}
In particular,
 $(F\circ G)^{(r+1)} = (G\circ F)^{(r+1)}=F^{(r+1)}+G^{(r+1)}$.
\end{rmk}

For a germ $F \in \cD$ satisfying $A(F) \neq 0$, we put $(S/A)(F)=S(F)/A(F)$.

\begin{lemma}
\label{lemma:2-flat}
Let $F_1$ and $F_2$ be $1$-flat germs in $\cD$
 with the opposite signs and satisfying $|(S/A)(F_1)|>|(S/A)(F_2)|$.
Then, for any neighborhood $\cV$ of the identity germ in $\cD$
 and any $\alpha,\beta \in \RR$,
 there exist $1$-flat germ $H \in \cV$
 and $m,n \geq 1$ such that the following holds:
\begin{gather*}
A(F_2^n)+A((H \circ F_1)^m) + \alpha =0,\\
S(F_1) \cdot \left\{S(F_2^n)+S((H \circ F_1)^m) + \beta\right\}>0.
\end{gather*}
\end{lemma}
\begin{proof}
Since $A(F_1) \cdot A(F_2) <0$,
 there exist sequences $(m_k)_{k=1}^\infty$
 and $(n_k)_{k=1}^\infty$ such that
 $\lim_{k \ra \infty}m_k =+\infty$ and
 $|m_k A(F_1)+n_k A(F_2)|<1$ for any $k \geq 1$.
Put $c_k=(m_k A(F_1)+n_k A(F_2)+\alpha)/(2m_k)$
 and let $H_k$ be the germ in $\cD$ given by
 $H_k(x)=x-c_kx^2+c_k^2x^3$.
Then, $A(H_k)=-2c_k$, $S(H_k)=0$,
 and $H_k$ converges to the identity germ in $\cD$.
Notice that, by Remark~\ref{rmk.tech}, we have
\begin{equation*}
 A(F_2^{n_k})+A((H_k \circ F_1)^{m_k})
  = m_k\left(A(F_1)+A(H_k)\right)+n_k A(F_2)
 = - \alpha
\end{equation*}
 and
\begin{align*}
\frac{S(F_2^{n_k})+S((H_k \circ F_1)^{m_k})}{m_k S(F_1)}
 & = \frac{n_k S(F_2)}{m_k S(F_1)}+\left(1+\frac{S(H_k)}{S(F_1)}\right)\\
 & =  -\frac{A(F_1)+A(H_k)+(\alpha/m_k)}{A(F_2)}\frac{S(F_2)}{S(F_1)}
 +1\\
 & \xrightarrow{k \ra \infty}
  1-\frac{A(F_1) S(F_2)}{A(F_2)S(F_1)} .
\end{align*}
Since $|S(F_1)/A(F_1)|>|S(F_2)/A(F_2)|$, the last term is positive.
This implies that
\begin{equation*}
 \lim_{k \ra \infty}S(F_1)\cdot \{S(F_2^{n_k})+S((H \circ F_1)^{m_k})\}
 =+\infty.
\end{equation*}
This shows that 
$H_k$, $m_k$ and $n_k$ satisfy the desired properties for sufficiently large $k$.
\end{proof}

%
%

\begin{lemma}
\label{lemma:(s+1)-flat}
Suppose $r \geq 2$.
Let $F_1$ and $F_2$ be $r$-flat germs in $\cD$
 such that $F_1^{(r+1)} \cdot F_2^{(r+1)}<0$.
Then, for any neighborhood $\cV$ of the identity
 and $\alpha \in \RR$,
 there exist $H \in \cD^r$ and $m,n \geq 1$ such that
\begin{equation*}
 F_2^m \circ(H \circ F_1)^n(x)=x+\alpha x^{r+1}+o(x^{r+1}).
\end{equation*}
\end{lemma}
\begin{proof}
Proof is similar to Lemma \ref{lemma:2-flat}.
Put $\alpha_i=F_i^{(r+1)}/(r+1)!$ for $i=1, 2$.
Since $\alpha_1 \cdot \alpha_2<0$,
 there exists sequences $(m_k)_{k=1}^\infty$ and $(n_k)_{k=1}^\infty$
 such that $m_k \to \infty$ as $k\to \infty$ and $|m_k \alpha_2+n_k \alpha_1| < 1$ holds.
Put
\begin{equation*}
 c_k=\frac{m_k \alpha_2+n_k \alpha_1-\alpha}{n_k}
\end{equation*}
 and $H_k(x)=x-c_kx^{r+1}$.
Then, $H_k$ converges to the identity in $\cD$.
Since $r$-flat germs are commutative up to $(r+1)$-st order,
 we have
\begin{align*}
F_2^{m_k} \circ  (H_k \circ F_1)^{n_k}
 & = x+(m_k \alpha_2+n_k (\alpha_1-c_k))x^{r+1} +o(x^{r+1})\\
 & =x+\alpha x^{r+1}+o(x^{r+1}).
\end{align*}
This shows that $H_k$, $m_k$ and $n_k$
 satisfy the desired property if $k$ is sufficiently large. 
\end{proof}


\begin{lemma}
\label{lemma:s>3-flat}
Suppose $r \geq 3$.
Let $F_1,\dots,F_4$ be $r$-flat germs in $\cD$.
Then, for any neighborhood $\cV$ of the identity map in $\cD$
 and $\alpha \in \RR$,
 there exist $H_1,\dots,H_4 \in \cV$
 and $n \geq 1$ such that
\begin{equation*}
 (H_4 \circ F_4)^n
  \circ \dots \circ (H_1 \circ F_1)^n(x)
 =x+\alpha x^{r+1}+o(x^{r+1}).
\end{equation*}
\end{lemma}
\begin{proof}
Set $c=(\sum_{i=1}^4 F_i^{(s+1)})/(s+1)!$.
Let $(G^t)_{t \in \RR}$ and
 $(H_\mu^t)_{t \in \RR}$ be the germs of local flows generated
 by vector fields $x^2\frac{\del}{\del x}$
 and $\mu x^s\frac{\del }{\del x}$, respectively.
They satisfy the following:
\begin{align*}
G^t(x) &
 =\frac{x}{1-tx}=x+tx^2+\cdots+ t^r x^{r+1} +o(x^{r+1}),\\
H_\mu^t(x) &
 =x+\mu tx^s+o(x^{r+1}).
\end{align*}
Since
\begin{align*}
G^t \circ H_\mu^t(x) & =(x+\mu tx^r)+t(x+\mu t x^r)^2+\dots+o(x^{r+1})\\
 & =G^t(x)+\mu t x^r+ 2\mu t^2 x^{r+1}+o(x^{r+1}),\\
H_\mu^t \circ G^t(x) & =(x+\mu tx^r)+\mu t(x+t x^2+\dots)^r+o(x^{r+1})\\
 & =G^t(x)+\mu t x^r+ r\mu t^2 x^{r+1}+o(x^{r+1}),
\end{align*}
 we have
\begin{equation*}
 G^t \circ H_\mu^t (x)=H_\mu^t(x) \circ G^t(x)
 -(r-2)\mu t^2 x^{r+1} +o(x^{r+1}).
\end{equation*}
This implies that
\begin{equation*}
 G^t \circ H_\mu^t \circ G^{-t} \circ H_\mu^{-t}(x)
 =x-(r-2)\mu t^2 x^{r+1}+o(x^{r+1}).
\end{equation*}

Set $\mu_n=c-(\alpha/n)$ and $t_n=1/\sqrt{n(r-2)}$ for $n \geq 1$.
The germs $G^{t_n}$ and $H_{\mu_n}^{t_n}$ converge
 to the identity in $\cD$.
Since $r$-flat germs commute with any germs in $\cD$
 up to $(r+1)$-st order,
 we have
\begin{align*}
\lefteqn{ (G^{t_n} \circ F_4)^n \circ (H_{\mu_n}^{t_n} \circ F_3)^n
 \circ  (G^{-t_n} \circ F_2)^n \circ (H_{\mu_n}^{-t_n} \circ F_1)^n}\\
 & =  G^{n t_n} \circ H_{\mu_n}^{n t_n}
   \circ G^{-n t_n} \circ H_{\mu_n}^{-n t_n}(x)
  +cn x^{r+1} +o(x^{r+1})\\
 & = x+\left\{-(r-2)\mu_n (n t_n)^2+cn \right\}x^{r+1}+o(x^{r+1})\\
 & =x+ (c-\mu_n)n\cdot x^{r+1}+o(x^{r+1})\\
 & =x+ \alpha x^{r+1}+o(x^{r+1}).
\end{align*}
Thus, letting $n$ large,
$H_1 = H_{\mu_n}^{-t_n}$,
$H_2 = G^{-t_n}$,
$H_3 = H_{\mu_n}^{t_n}$ and
$H_4 = G^{t_n}$, we complete the proof.
\end{proof}

These lemmas allow us to construct flat germs.
For the practical use, we need to take their realizations
 as close to identity diffeomorphisms. The following statement
 shows that such realization are always possible
 (we omit the proof since it is well-known):
\begin{rmk}\label{rmk:realization}
For any neighborhood $\mathcal{N}$
 of the identity map in $\mathrm{Diff}^{\infty}([-1, 1])$,
 any $x \in (0,1)$,
 and any neighborhood $V \subset [-1, 1]$ of $x$,
 there exists a neighborhood $\mathcal{M}$
 of the identity germ in $\mathcal{D}$ such that
 for every $F \in \mathcal{M}$ there exists a diffeomorphism
 $\tilde{F} \in \mathcal{N}$ such that $\tilde{F}(x)=x$,
 $[\tilde{F}]_x=F$, and $\mathrm{supp} \tilde{F} \subset V$.
 \end{rmk}

%

\section{Connecting lemmas}\label{sec:connecting}
In this section,
 we show that, in the presence of a blender,
 we can create an orbit connecting any two prescribed points
 by a small perturbation.
We consider semigroup actions
 $\rho=(f_0,f_1,f_2) \in \cA^r(\{0,1,2\})$ (where $r \geq 1$)
which satisfy the following conditions:
\begin{enumerate}
\item $(f_1,f_2) \in \cA^r(\{1,2\})$
 is a blender on a closed interval $J \subset [0,1]$.
\item $\mathrm{Int}(J) \cap f_0(\mathrm{Int}(J))\neq \emptyset$.
\end{enumerate}
Notice that these conditions hold for $\rho \in \mathcal{W}^r$.

We prepare several definitions. For $x\in [0, 1]$, we put
\begin{align*}
\cO_-(x,(f_1,f_2)) & =\{y \in [0,1] \mid \rho^\omega(y)=x
 \text{ for some }\omega \in \{1,2\}^*\}.
\end{align*}
We say that a point $x \in J$ is {\it $(f_1,f_2)$-generic} if
 the closure of $\cO_-(x,(f_1,f_2))$ contains $J$.
An $(f_1,f_2)$-generic point is a generic point in the sense of Baire as well.
Indeed, the set
 $J(U)=J \cap \bigcup_{\omega \in \{1,2\}^*}\rho^\omega(U)$
 is an open and dense subset of $J$ for any non-emptyopen subset $U$ of $J$.
 Take a countable open basis $(U_n)_{n \geq 1}$ of $J$.
Since $(f_1,f_2)$ is a blender on $J$,
 every point in the residual subset $\bigcap_{n \geq 1}J(U_n)$ of $J$
 is $(f_1,f_2)$-generic.

For $h \in \Diff^r([0,1])$,
 we define an element $\rho_h$ of $\cA^r(\{0,1,2\})$ by
\begin{equation*}
\rho_h=(h \circ f_0,f_1,f_2).
\end{equation*}
For a point $x \in [0,1]$ and a word $\omega = s_n\cdots s_1 \in \{0, 1, 2\}^{\ast}$, set
\begin{equation*}
\Sigma_h^\omega(x)=\{\rho_h^{s_k\dots s_1}(x)
 \mid s_k = 0, \, k=1,\dots,n \}.
\end{equation*}
When $h$ is the identity map, we simply write $\rho^\omega$ and $\Sigma^\omega$ for $\rho_h^\omega$ and $\Sigma_h^\omega$.
If $h \in \Diff^r([0,1])$ satisfies $\supp(h) \cap \Sigma^\omega(x) =\emptyset$, then
$[\rho_{h}^{\omega}]_x =[\rho^{\omega}]_x$, hence $[\rho_{h}^{0 \omega}]_x  = [h]_{\rho^{0\omega}(x)} \circ [\rho^{0 \omega}]_x$
(remember that by $[\, \cdot \, ]_x$ we denote the germ of a diffeomorphism at $x$, see Section~\ref{sec:cancel}).
For $\omega \in \{0,1,2\}^*$, a point $x \in [0,1]$ is {\it $\omega$-periodic} for $\rho$ if $\rho^\omega(x)=x$.
For $r \in [1, +\infty]$, we say that a $\omega$-periodic point $x$ is {\it $r$-flat} if the germ $[\rho^\omega]_x$ is $r$-flat.
For a word $\omega \in S^*$ with an alphabet $S$, we denote the length of $\omega$ by $|\omega|$,
i.e. $|\omega|=n$ if $\omega=s_n\dots s_1$. For $1\leq k \leq n$, we set $\omega|_{k} = s_k\cdots s_1$.

\begin{lemma}
\label{lemma:connecting 1}
Let $p$ and $p'$ be points in $J$ such that $p'$ is $(f_1,f_2)$-generic.
For any neighborhood $\cN \subset \Diff^\infty([0,1])$ of the identity map, any non-empty open subset $U$ of $J \cap f_0(J)$
and any $l \geq 1$, there exist $h \in \cN$ and $\omega \in \{0,1,2\}^*$ such that $\supp(h) \cup \Sigma_h^\omega(p) \subset U$,
$\rho_h^\omega(p)=p'$, and $|\omega| \geq l$.
\end{lemma}
\begin{proof}
Since $p'$ is $(f_1,f_2)$-generic,
 $\cO_-(p',\rho)$ intersects with $U \subset J \cap f_0(J)$.
Take $\eta \in \{1,2\}^*$ such that
 $|\eta| \geq l$ and $q=(\rho^{\eta})^{-1}(p')$ belongs to $U$.
Since $(f_1,f_2)$ is a blender on $J$,
 the point $f_0^{-1}(q)$ is contained in the closure of $\cO_+(p,\rho)$.
Hence, there exist $h \in \cN$ and $\eta' \in \{1,2\}^*$
 such that $\supp(h) \subset U$ and $h \circ f_0 \circ \rho^{\eta'}(p)=q$.
Put $\omega=\eta 0 \eta'$.
 Then, $\rho_h^\omega(p)=p'$, $\Sigma_h^\omega(p) =\{q\} \subset U$,
 and $|\omega| \geq l$.
\end{proof}

\begin{rmk}\label{rmk:solo}
Notice that in the word $\omega$ obtained in the above proof the letter $0$ appears only once.
\end{rmk}

The next lemma shows that
 when there is an $r$-flat periodic point somewhere in $J$,
 then the connecting orbit can be constructed
 in such a way that the germ $[\rho_h^\omega]_{p}$ will coincide with
 any prescribed one up to order $r$. In particular,
 we can construct the connecting orbit
 for which the corresponding germ will be $r$-flat.

\begin{lemma}
\label{lemma:connecting 2}
Let $p$ and $p'$ be points in $J$ such that $p'$ is $(f_1,f_2)$-generic.
 Suppose that there exist an $(f_1,f_2)$-generic point $\hat{p} \in J$,
 a word $\gamma \in \{0,1,2\}^*$, and $r_0 \in [1, r]$
 such that $\hat{p}$ is an $r_0$-flat $0\gamma$-periodic point of $\rho$
 and $\hat{p} \not\in \Sigma^\gamma(\hat{p})$.
Then,
 for any neighborhood $\cN \subset \Diff^\infty([0,1])$ of the identity map,
 any neighborhood $V$ of $\hat{p}$,
 any non-empty open subset $U$ of $J \cap f(J)$,
 any germ $F \in \cD$ and any $l \geq 1$,
 there exist $\omega \in \{0,1,2\}^*$ and $h \in \cN$
 such that $|\omega| \geq l$, $\supp(h) \subset U \cup V$,
 $\rho_h^\omega(p)=p'$, $[\rho_h^\omega(t)]_p=F(t)+o(t^{r_0})$,
 and $\Sigma_h^\omega(p) \subset \Sigma^\gamma(\hat{p}) \cup U \cup V$.
\end{lemma}
\begin{proof}
Without loss of generality, we may assume that the sets $U$, $V$,
 and $\Sigma^\gamma(\hat{p})$ are mutually disjoint.
Applying Lemma~\ref{lemma:connecting 1}
 for pairs $(p,\hat{p})$, $(\hat{p},p')$ and the open set $U$,
 we obtain $\omega_1, \omega_2 \in \{1,2\}^*$ and $\bar{h} \in \cN$
 such that $|\omega_1| \geq l$, $\supp(\bar{h}) \subset U$,
 $\rho_{\bar{h}}^{\omega_1}(p)=\hat{p}$,
 $\rho_{\bar{h}}^{\omega_2}(\hat{p})=p'$, and
 $\Sigma_{\bar{h}}^{\omega_1}(p) \cup \Sigma_{\bar{h}}^{\omega_2}(\hat{p})
 \subset U$.
More precisely, we first take $U_1, U_2 \subset U$ satisfying
 $U_1 \cap U_2 = \emptyset$.
Then we apply Lemma~\ref{lemma:connecting 1} for $(p, \hat{p})$ and $U_1$,
 and $(\hat{p}, p')$ and $U_2$ to obtain two diffeomorphisms
 $h_1$ and $h_2$ respectively.
Then, since they have disjoint support, their composition
 $\bar{h} = h_1 \circ h_2$ gives us the desired $\bar{h}$.

Put $F_1=[\rho_{\bar{h}}^{\omega_1}]_{p}$
 and $F_2=[\rho_{\bar{h}}^{\omega_2}]_{\hat{p}}$.
By Lemma \ref{lemma:exp},
 there exists a one-parameter family of germs $(\phi^t)_{t \in \RR}$ in $\cD$
 (which is a one-parameter group up to order $r_0$)
 such that
\begin{equation*}
\phi^1(x)=(F_2)^{-1} \circ F \circ F_1^{-1}(x)+o(x^{r_0}). 
\end{equation*}
Notice that, by Remark~\ref{rmk:realization},
for each $\phi^{1/N}$,
we can choose a diffeomorphism $\varphi^{1/N} :[0, 1]\to [0, 1]$, such that
$\varphi^{1/N}(\hat{p})=\hat{p}$ and
 $[\varphi^{1/N}]_{\hat{p}}=\phi^{1/N}$.
 Furthermore, by choosing $N$ sufficiently large,
we can assume that
 $h=\varphi^{1/N} \circ \bar{h}$ is contained in $\cN$,
 being the support of $\varphi^{1/N}$ arbitrarily
 close to the point $\{\hat{p}\}$.

Then, the support of $h$ is contained in $U \cup V$,
 and hence, it does not intersect $\Sigma^\gamma(\hat{p})$.
Put $\omega=\omega_2(0\gamma)^N\omega_1$.
Since $\hat{p}$ is an $r_0$-flat $0 \gamma$-periodic point,
 we have
\begin{align*}
 [\rho_h^{\omega}(x)]_p
 & = F_2 \circ
 ([\varphi^{1/N}]_{\hat{p}} \circ [\rho^\gamma]_{\hat{p}})^N \circ F_1(x)\\
 & = F_2 \circ [\varphi^1]_{\hat{p}} \circ F_1(x) +o(x^{r_0})\\
 & = F(x)+o(x^{r_0}).
\end{align*}
We can also see that $|\omega| \geq l$ and
 $\Sigma_h^{\omega}(p)
 =\Sigma^\gamma(\hat{p}) \cup \{\hat{p}\} \cup \Sigma_{h_1}^{\omega_1}(p)
 \cup \Sigma_{h_2}^{\omega_2}(\hat{p})$.
The latter implies that
 $\Sigma_h^{\omega}(p) \subset \Sigma^\gamma(\hat{p}) \cup U \cup V$.
\end{proof}
\begin{rmk}\label{rmk:solo2}
In this lemma, because of Remark~\ref{rmk:solo}, we can assume that there exists a point
$y \in \Sigma_h^{\omega}(p)\cap U$ and a unique integer $k \geq 1$ such that
$\rho^{\omega|_k}_h(p) =y$ and $\omega_k=0$,
where $\omega= \omega_{|\omega|}\cdots \omega_1$.
Indeed, the point in $U_1 \cap \Sigma_h^{\omega}(p)$ is such a point.
Roughly speaking, $y$ is a point which appears in $\Sigma_h^{\omega}(p)$
only once.
\end{rmk}

Below we will use the following perturbation result whose proof we omit.
\begin{rmk}\label{rem:flat}
Let $f \in \mathrm{Diff}^r([0, 1])$ where $r \in [1, +\infty]$.
If $f$ has an $r$-flat fixed point $x$, then $C^r$-arbitrarily close to
the identity there exists
$g\in \mathrm{Diff}^{\infty}([0, 1])$
whose support is contained in an arbitrarily small neighborhood of
$\{x\}$ such that $g \circ f$ coincides with the identity map
near $x$.
\end{rmk}
For the case $r= \infty$, we have the following
\begin{lemma}
\label{lemma:connecting inf}
Under the assumptions of Lemma~\ref{lemma:connecting 2},
we furthermore assume that $\rho \in \mathcal{A}^{\infty}(\{0, 1,2  \})$
and $\hat{p}$ is $\infty$-flat $0\gamma$-periodic point of $\rho$.
Then, for any neighborhood $\cN \subset \Diff^\infty([0,1])$ of
the identity map,
any neighborhood $V$ of $\hat{p}$,
any non-empty open subset $U$ of $J \cap f(J)$,
there exist $\omega \in \{0,1,2\}^*$ and $h \in \cN$
such that $\supp(h) \subset U \cup V$,
 $\rho_h^\omega(p)=p'$, $[\rho_h^\omega(t)]_p= [t]$,
and
$\Sigma_h^\omega(p) \subset \Sigma^\gamma(\hat{p}) \cup V \cup U$.
\end{lemma}
\begin{proof}
For a given  $\cN \subset \Diff^\infty([0,1])$, since the
$C^{\infty}$-topology coincides with the projective limit of
$C^r$-topology, there exists $s \geq 1$ and
$\mathcal{N}^{s} \subset \Diff^s([0,1])$ such that
$\mathcal{N}^{s} \cap \Diff^{\infty}([0,1]) \subset \mathcal{N}$.
We fix such $s$ and $\mathcal{N}^s$.

Then, for given $\rho$, $U$ and $V$, by applying
Lemma~\ref{lemma:connecting 2} where we let $F$ be the identity germ and $r_0 = s$,
we take $h$ which is arbitrarily $C^{\infty}$-close to the identity map
and $\omega \in \{ 0, 1, 2\}^{\ast}$ such that
$\rho_{h}^{\omega}(p) = p'$ and
$[\rho_{h}^{\omega}(t)]_{p} = t + o(t^s)$.

Now, by Remark~\ref{rem:flat}, we choose
$\tilde{h} \in \mathrm{Diff}^{\infty}([0, 1])$ supported in an arbitrarily
small neighborhood of $p'$ such that it is arbitrarily
$C^{s}$ close to the identity and
$[\tilde{h} \circ \rho_{h}^{\omega}(t)] =[t]$.
We take a point $y \in \Sigma_h^{\omega}(p) \cap U$
given by Remark~\ref{rmk:solo2}
and choose $\omega_y$, $\omega'_y$ such
that $\rho^{\omega_y}_h(p)=y$ and
$\omega = \omega'_y \omega_y$.

Now, we take $\bar{h}
= (\rho^{\omega'_y}_{h})^{-1} \circ \tilde{h} \circ \rho^{\omega'_y}_{h}$. Then,  by taking
$\tilde{h}$ arbitrarily $C^s$-close to the identity, we can assume
that $\bar{h} \in \mathcal{N}^{s} \cap \Diff^{\infty}([0,1])$.
Furthermore, by shrinking the support of $\tilde{h}$, we can check that
$\rho^{\omega}_{\bar{h}\circ h}(p) = p'$ keeping
$[\rho^{\omega}_{\bar{h}\circ h}(t)]_p = [t]$ and
all the other conditions un tact.
Thus, $\bar{h}\circ h$ gives us the desired map $g$.
\end{proof}

%

\section{Creation of $r$-flat periodic orbits}\label{sec:flat}
In this section, we prove the following proposition which implies Theorem
 \ref{thm:arbitrary growth}.
\begin{prop}
\label{prop:flat periodic points}
For any $\rho=(f_0,f_1,f_2) \in \cW^\infty$,
 any open neighborhood $\cN$ of the identity map in $\Diff^\infty([0,1])$,
 any $N,l \geq 1$, and $r \in[1, \infty]$,
 there exist $h \in \cN$,
 $(f_1,f_2)$-generic distinct points $\hat{p}_1,\dots,\hat{p}_N$ in $J$,
 and $\gamma_1,\dots,\gamma_N \in \{0,1,2\}^*$
 such that $|\gamma_i| \geq l$,
 $\hat{p}_i$ is an $r$-flat $0 \gamma_i$-periodic point for $\rho_h$
 for any $i=1,\dots,N$,
 and the sets $\{\hat{p}_1,\dots,\hat{p}_N\}$,
 $\Sigma^{\gamma_1}(\hat{p}_1), \dots, \Sigma^{\gamma_N}(\hat{p}_N)$
 are mutually disjoint.
\end{prop}
Let us, first, see how we derive Theorem \ref{thm:arbitrary growth}
from Proposition \ref{prop:flat periodic points}. In the proof,
we use perturbations given by the following construction.
\begin{rmk}\label{rem:perper}
Let $f \in \mathrm{Diff}^r([0, 1])$ where $1 \leq r \leq + \infty$.
If $f$ coincides with the identity map on some
non-empty open interval $U$,
then for every $\ell >0$,
there exists $g \in \mathrm{Diff}^{\infty}([0, 1])$
which is arbitrarily $C^{\infty}$-close to the identity and is
supported in an arbitrarily
small interval in $U$, such that
$g \circ f$ has more than $\ell$ attracting fixed points in $U$.
For instance, one can build such $g$ as follows:
let $g$ be a map which has the form $x + a \sin (k x)$
($a$ and $k$ are some constants).
Collapse it to the identity map outside $U$ by some
bump function. Then,
by choosing $a$ and $k$ appropriately,
one can see that $g$ will be
arbitrarily $C^{\infty}$-close to the identity and have
an arbitrarily large number of attracting periodic points in $U$.
The details is left to the reader.
\end{rmk}

Another important remark is that $\cW^\infty$ is dense in $\cW^r$
for any $r \geq 1$. For $r \geq 3$, it is trivial, since
$\cA^\infty(\{0,1,2\})$ is dense in $\cA^r(\{0,1,2\})$.
So, let us consider the case $r=2$. For any $\rho=(f_0,f_1,f_2) \in \cW^2$
 with heteroclinic points $z_1$ and $z_2$ satisfying Sign Condition I,
 a $C^2$-small (but $C^3$-large) perturbation of $f_0$ at $z_1$ and $z_2$
 creates a $C^\infty$ map $f$ such that
 $\tau_S(z_1,f) \cdot \tau_S(z_2,f)<0$.
This implies that $\cW^\infty$ is $C^2$-dense in $\cW^2$.
Similarly, we can see that $\cW^\infty$ is $C^1$-dense in $\cW^1$.

\begin{proof}
[Proof of Theorem \ref{thm:arbitrary growth}
 from Proposition \ref{prop:flat periodic points}]
Fix $1 \leq r \leq \infty$
 and a sequence $(a_n)_{n=1}^\infty$ of integers.
Put
\begin{equation*}
 \cU(\omega)=\{\rho \in \cW^r \mid
 \#\Fix_a(\rho^\omega) \geq |\omega|\cdot a_{|\omega|}\}
\end{equation*}
 for $\omega \in \{0,1,2\}^*$
 and $\cU_n=\bigcup_{|\omega|\geq n}\cU(\omega)$ for $n \geq 1$.
By the persistence of attracting periodic points,
$\cU(\omega)$ is open for every $\omega$ and, accordingly,
$\mathcal{U}_n$ is open as well.
Notice that every $\rho \in \bigcap_{n \geq 1}\cU_n$ satisfies
\begin{equation*}
 \limsup_{n \ra \infty}
 \frac{\sum_{\omega \in \{0,1,2\}^n}\Fix_a(\rho^\omega)}{a_n}
 =\infty.
\end{equation*}
Hence, it is sufficient to show that $\cU_n$ is a dense subset of $\cW^r$
 for every $n \geq 1$.

Fix a non-empty open subset $\cU$ of $\cW^r$,
 $\rho=(f_0,f_1,f_2) \in \cU \cap \cW^\infty$, and $n \geq 1$.
Take a neighborhood $\cN^{1/3}$ of the
identity map in $\Diff^\infty([0,1])$
 such that $((h_3 \circ h_2 \circ h_1) \circ f_0, f_1,f_2) \in \cU$
 for any $h_1,h_2,h_3 \in \cN^{1/3}$.
By Proposition \ref{prop:flat periodic points},
 there exist $h_1 \in \cN^{1/3}$, $p \in J$,
 and $\gamma \in \{0,1,2\}^*$ such that $|\gamma| \geq n$
 and $p$ is an $r$-flat $0\gamma$-periodic point of $\rho_{h_1}$
 satisfying $p \not\in \Sigma_{h_1}^\gamma(p)$.
Take $h_2 \in \cN^{1/3}$ such that
 $\supp(h_2) \cap \Sigma_{h_1}^\gamma(p)=\emptyset$,
 and $\rho_{h_2 \circ h_1}^{0\gamma} = h_2 \rho_{h_1}^{0\gamma}$
 is the identity map on a small neighborhood $V$ of $p$
 (see Remark~\ref{rem:flat}).
We also take $h_3 \in \cN^{1/3}$ such that
 $\supp(h_3) \cap \Sigma^\gamma(p)=\emptyset$
 and $\rho^{0\gamma}_{h_3\circ h_2 \circ h_1}$ admits
 more than $a_{|0\gamma|}|0\gamma|$ attracting fixed points in $V$
 (see Remark~\ref{rem:perper}).
Now $((h_3 \circ h_2 \circ h_1) \circ f_0, f_1,f_2)$
 is contained in $\cU \cap \cU_n$.
Since the choice of $\cU$ is arbitrary, the set $\cU_n$ is dense in $\cW^r$.
\end{proof}

Let us prove Proposition~\ref{prop:flat periodic points}.
The proof is done by several inductive steps.
The following notation will be used throughout this Section.
Let $\rho=(f_0,f_1,f_2) \in \cA^\infty(\{0,1,2\})$
and $J \subset [0,1]$ be a closed interval
such that $(f_1,f_2)$ is a blender on $J$.
We assume that $f_0$ has a repeller-attractor pair $(p,q)$ in $\Int J$.

In the following first step,
we create $1$-flat periodic points
 which satisfy certain estimates on $A$ and $S$.
Recall that for a germ $F \in \cD$, we denote
$(S/A)(F) = S(F)/A(F)$.
\begin{lemma}
\label{lemma:connecting fl}
In the above setting,
 assume that for a repeller-attractor pair $(p, q)$
 the repeller $p$ is $(f_1,f_2)$-generic,
  $(\log f_0'(p) ) / (\log f'_0(q)) $ is irrational,
 and the pair $(p,q)$ has a heteroclinic point $z_*$
 with $\tau_A(z_*,f_0) \neq 0$ and $\tau_S(z_*,f_0) \neq 0$.
Then,
 for any neighborhood $\cN$ of the identity map in $\Diff^\infty([0,1])$ and
 any finite subset $\Lambda$ of $[0,1]$ and $\nu >0$,
 there exist $h \in \cN$, an $(f_1,f_2)$-generic point
 $\hat{p} \in J \setminus \Lambda$, and
 $\gamma \in \{0,1,2\}^*$ such that
 $\hat{p}$ is a $1$-flat $0\gamma$-periodic point for $\rho_h$,
 $(\supp (h) \cup \Sigma_h^\gamma(\hat{p})) \cap \Lambda =\emptyset$,
 the sign of the germ $[\rho_h^{0\gamma}]_{\hat{p}}$ is
 $(\tau_A(z_*,f_0),\tau_S(z_*,f_0))$, and
 $|(S/A)([\rho_h^{0\gamma}]_{\hat{p}})|>\nu$.
\end{lemma}
\begin{proof}
We may assume that $\Lambda \cap \{f_0^i(z_*)\mid i \in \ZZ\}=\emptyset$
by replacing $z_*$ with a nearby point if necessary.

We denote
\[
X =  f_0^{-1}(\Lambda) \cup \Lambda \cup f_0(\Lambda) \cup [p,q],
\]
where $[p, q]$ denotes the closed interval whose ends are $p$ and $q$
(note that we do not assume $p$ is to the left of $q$).
Notice that $\{ f^i(z_{*}) \mid i \in \mathbb{Z} \} \subset [p, q]$.
Since $\Int J$ contains a repelling fixed point $p$ of $f_0$,
we see that
$ (\mathrm{Int}(J) \cap f_0(\mathrm{Int}(J)) \cap
f_0^{-1}(\mathrm{Int}(J))) \setminus X
\neq \emptyset$.
Take a point $\hat{p}$ in this set
such that $\hat{p}$ is $(f_1,f_2)$-generic
and $f_0(\hat{p}) \neq \hat{p}$.
Remark that the three points $\hat{p}$, $f_0(\hat{p})$,
 and $f_0^{-1}(\hat{p})$
 are mutually distinct (since $f_0$ is orientation-preserving).
Also, notice that by construction
$\{\hat{p}, f_0(\hat{p}), f_0^{-1}(\hat{p})\}
\cap (\Lambda \cup [p, q]) =\emptyset$.
We take an open neighborhood $U$ of
$\{\hat{p}, f_0(\hat{p}), f_0^{-1}(\hat{p})\}$
in  $
(\mathrm{Int}(J) \cap f_0(\mathrm{Int}(J)) \cap
f_0^{-1}(\mathrm{Int}(J))) \setminus (\Lambda \cup [p, q])$.
Then we take
a neighborhood $V$ of $(f_0(\hat{p}),\hat{p})$ in $U \times U$,
 a real number $\epsilon>0$,
 and a continuous family $(h_{s,v})_{s \in (-\epsilon,\epsilon), v \in V}$
 in $\cN$ such that
\begin{enumerate}
\item $\supp(h_{s,(x,y)}) \subset U$,
\item $h_{s,(x,y)}(f_0(\hat{p}))=x$, $(h_{s,(x,y)})'(f_0(\hat{p}))=1$,
\item $h_{s,(x,y)}(y)=\hat{p}$,
 and $(h_{s,(x,y)})'(y)=e^s$,
\end{enumerate}
 for any $s \in (-\epsilon,\epsilon)$ and $(x,y) \in V$.
Since, by construction,
$f_0(\hat{p}), f^{-1}_0(\hat{p}) \in J$
and $p$ is $(f_1,f_2)$-generic,
 there exist $\eta,\eta' \in \{1,2\}^*$ such that
 $\bar{v}=((\rho^{\eta})^{-1}(p), f_0 \circ \rho^{\eta'}(q))$
 is contained in $V$.
Let $\gamma_{m,n}  = 0\eta'0^{m+n}\eta 0$ and
\begin{align*}
v_{m,n} & =
 ((\rho^\eta)^{-1}(f_0^{-m}(z_*)), f_0 \circ \rho^{\eta'}(f_0^n(z_*)))
\end{align*}
 for $m,n \geq 1$.
Since $f_0^{n}(z_*)$ converges to $q$ and $f_0^{-n}(z_*)$
converges to $p$ as $n \ra \infty$,
 $v_{m,n}$ converges to $\bar{v}$ as $m,n \ra \infty$.
Fix $N \geq 1$ such that $v_{m,n} \in V$ for any $m,n \geq N$.
For any $m,n \geq N$, we have
\begin{equation}
\label{eqn:connect 1}
\rho_{h_{s,v_{m,n}}}^{\gamma_{m,n}}  =
 (h_{s,v_{m,n}} \circ f_0)
 \circ \rho^{\eta'} \circ f_0^{m+n} \circ \rho^\eta
 \circ (h_{s,v_{m,n}} \circ f_0)
\end{equation}
 in a small neighborhood of $\hat{p}$.
In particular, $\rho_{h_{s,v_{m,n}}}^{\gamma_{m,n}}(\hat{p})=\hat{p}$;
see Figure~\ref{fig:connect 1}.
\begin{figure}
\begin{center}
\includegraphics[scale=0.8]{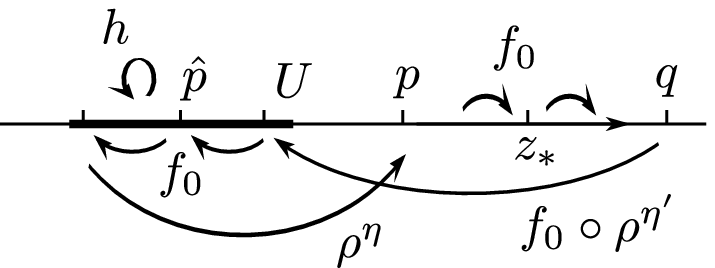}
\caption{Proof of Lemma \ref{lemma:connecting fl}}
\label{fig:connect 1}
\end{center}
\end{figure}

We will show that $\gamma=\gamma_{m,n}$ and $h=h_{s,v_{m,n}}$
 satisfy the required properties at $\hat{p}$
 for suitably chosen $m, n$ and $s$.
Let $\varphi:W^u(p) \ra \RR$ and $\psi:W^s(q) \ra \RR$ be
 linearizations of $f_0$ at $p$ and $q$ respectively.
Notice that since $\mathrm{supp}\, (h_{s, (x, y)}) \subset U$,
the perturbation by $h_{s, (x, y)}$ does not change the
behavior of $\varphi$ and $\psi$ on $(p, q)$.
Let $\lambda_p=f_0'(p)$, $\lambda_q=f_0'(q)$, and define a continuous function
\begin{align*}
 c(x,y) =
 (f_0 \circ \rho^{\eta'} \circ \psi^{-1})'
 ((f_0 \circ \rho^{\eta'} \circ \psi^{-1})^{-1}(y))
 \cdot
  (\varphi \circ \rho^\eta)'(x)
 \cdot f_0'(\hat{p})
\end{align*}
 on $V$.
By equality (\ref{eqn:connect 1}), notice that the following holds:
\begin{equation*}
(\rho_{h_{s,v_{m,n}}}^{\gamma_{m,n}})'(\hat{p})
 = e^s \cdot  \lambda_p^m \cdot \lambda_q^n \cdot c(v_{m,n})
 \cdot (\psi \circ \varphi^{-1})'(\varphi(z_*)).
\end{equation*}
 for any $m,n \geq N$.
Since $\lambda_p>1>\lambda_q$ and
 the ratio $\log\lambda_p/\log \lambda_q$ is irrational,
 there exist increasing sequences $(m_k)_{k \geq 1}$ and $(n_k)_{k \geq 1}$
 such that
\begin{equation*}
 \lim_{k \ra \infty} \lambda_p^{m_k} \cdot \lambda_q^{n_k} \cdot c(\bar{v})
 \cdot (\psi \circ \varphi^{-1})'(\varphi(z_*))=1.
\end{equation*}
By the continuity of the function $c$, the sequence $c(v_{m_k,n_k})$ converges to
 $c(\bar{v})$ as $k\ra \infty$.
Hence, we can choose a converging to zero sequence of real numbers
$(s_k)_{k \geq 1}$
 such that $|s_k|<\epsilon$ and
\begin{equation*}
(\rho_{h_{s_k,v_{m_k,n_k}}}^{\gamma_{m_k,n_k}})'(\hat{p})
= e^{s_k} \cdot \lambda_p^{m_k} \cdot \lambda_q^{n_k} \cdot c(v_{m_k,n_k})
 \cdot (\psi \circ \varphi^{-1})'(\varphi(z_*))
=1
\end{equation*}
 for all large $k$.

 Let us estimate
 $A([\rho_{h}^{0\gamma}]_{\hat{p}})$ and
 $S([\rho_h^{0\gamma}]_{\hat{p}})$
 for $(\gamma,h)=(\gamma_{m_k,n_k},h_{s_k,v_{m_k,n_k}})$.
Let
\begin{align*}
 F_k & =[\varphi \circ \rho^{\eta}
 \circ (h_{s_k,v_{m_k,n_k}} \circ f_0)]_{\hat{p}},\\
 G_k & =[(h_{s_k,v_{m_k,n_k}} \circ f_0) \circ \rho^{\eta'}
 \circ \psi^{-1}]_{\psi \circ f_0^{n_k}(z_*)}.
\end{align*}
Then, $F_k$ and $G_k$ converge to
 $[\varphi \circ \rho^{\eta} \circ  (h_{0,\bar{v}} \circ f_0)]_{\hat{p}}$
 and $[(h_{0,\bar{v}} \circ f_0) \circ \rho^{\eta'} \circ \psi^{-1}]_{\psi(q)}$
 respectively, and
\begin{equation*}
\left[\rho_{h_{s_k,v_{m_k,n_k}}}^{0\gamma_{m_k,n_k}}\right]_{\hat{p}}
 =G_k \circ L_{\lambda_q}^{n_k} \circ [\psi \circ \varphi^{-1}]_{\varphi(z_*)}
 \circ L_{\lambda_p}^{m_k} \circ F_k,
\end{equation*}
 where $L_\lambda$ is the germ of the map $x \mapsto \lambda x$ at $0$.
By the cocycle property of $A(\,\cdot\,)$
 and the equality
 $(\rho_{h_{s_k,v_{m_k,n_k}}}^{0\gamma_{m_k,n_k}})'(\hat{p})=1$,
 together with the obvious relation $A(L_{\lambda})=0$,
we have
\begin{equation*}
A([\rho_{h_{s_k,v_{m_k,n_k}}}^{0\gamma_{m_k,n_k}}]_{\hat{p}})
 =A(G_k) \cdot ((G_k)')^{-1}
 + A([\psi\circ \varphi^{-1}]_{\varphi(z_*)})
 \cdot \lambda_p^{m_k} (F_k)'
 + A(F_k).
\end{equation*}
This implies that
\begin{equation*}
\lim_{k \ra \infty}
 \lambda_p^{-m_k} \cdot
 A([\rho_{h_{s_k,v_{m_k,n_k}}}^{0\gamma_{m_k,n_k}}]_{\hat{p}})
 =A([\psi \circ \varphi^{-1}]_{\varphi(z_*)}).
\end{equation*}
Similarly, we have
\begin{equation*}
\lim_{k \ra \infty}
 \lambda_p^{-2m_k} \cdot
 S([\rho_{h_{s_k,v_{m_k,n_k}}}^{0\gamma_{m_k,n_k}}]_{\hat{p}})
 =S([\psi \circ \varphi^{-1}]_{\varphi(z_*)}).
\end{equation*}
Therefore, the germ
 $[\rho_{h_{s_k,v_{m_k,n_k}}}^{0\gamma_{m_k,n_k}}]_{\hat{p}}$
 has the same sign as $[\psi \circ \varphi^{-1}]_{\varphi(z_*)}$
 and $|(S/A)[\rho_{h_{s_k,v_{m_k,n_k}}}^{0\gamma_{m_k,n_k}}]_{\hat{p}}|>\nu$
 for all large $k$.

Thus we have completed the construction of
$h= h_{s_k,v_{m_k,n_k}}$ and $\hat{p}$ having
all the desired properties.
\end{proof}

\begin{rmk}\label{rmk:length}
In the proof of Lemma~\ref{lemma:connecting fl},
we can assume that the length of $\gamma$ is arbitrarily long.
Indeed, we only need to choose one which corresponds to a large $k$.
\end{rmk}

On the second step, we construct
a $2$-flat periodic point from five $1$-flat periodic points.
\begin{lemma}
\label{lemma:2-flat periodic}
Suppose that
 there exist mutually distinct $(f_1,f_2)$-generic points
 $p_1,\dots,p_5$ in $J$
 and $\gamma_1,\dots, \gamma_5 \in \{ 0,1,2\}^*$ such that
\begin{enumerate}
 \item each $p_i$ is a $1$-flat $0 \gamma_i$-periodic point,
 \item $\Sigma^{\gamma_1}(p_1), \dots, \Sigma^{\gamma_5}(p_5)$
 and $\{p_1,\dots,p_5\}$ are mutually disjoint,
 \item $[\rho^{0 \gamma_4}]_{p_4}$ and
  $[\rho^{0 \gamma_5}]_{p_5}$ have opposite signs, and
\begin{equation*}
|(S/A)([\rho^{0 \gamma_4}]_{p_4})|
 >|(S/A)([\rho^{0 \gamma_5}]_{p_5})|.
\end{equation*}
\end{enumerate}
Then, for any neighborhood $V$ of $\{p_1,\dots,p_5\}$,
any neighborhood $\cN$ of the identity map in $\Diff^\infty([0,1])$,
 and any non-empty open subset $U$ of $J \cap f_0(J)$,
 there exist an $(f_1,f_2)$-generic point $\hat{p} \in U$,
 $\hat{\gamma} \in \{0,1,2\}^*$, and $h \in \cN$ such that
 $\supp(h) \subset U \cup V$,
 $\hat{p}$ is a $2$-flat $0 \hat{\gamma}$-periodic point for $\rho_h$
 with $\hat{p} \not\in \Sigma_{h}^{\hat{\gamma}}(\hat{p})$,
 $\Sigma_{h}^{\hat{\gamma}}(\hat{p})
 \subset \bigcup_{i=1}^5 \Sigma^{\gamma_i}(p_i) \cup U \cup V$,
 and
 $\sgn (S([\rho_h^{0 \hat{\gamma}}]_{\hat{p}}))
 =\sgn (S([\rho^{0 \gamma_4}]_{p_4}))$.
\end{lemma}
\begin{proof}
Without loss of generality,
 we may assume that $U$, $V$ and $\bigcup_{i=1}^5\Sigma^{\gamma_i}(p_i)$
 are mutually disjoint.

We prepare several regions for the construction of the periodic orbit.
Take a neighborhood $\cN^{1/4}$ of the identity in $\Diff^{\infty}([0,1])$
 such that any composition of four diffeomorphisms
 in $\cN^{1/4}$ belongs to $\cN$.
 We take a neighborhood $V_i$ of $p_i$ for each $i=1,\dots,5$
 such that $V_1,\dots,V_5$ are mutually disjoint subsets of $V$.
 Then, we take non-empty, mutually disjoint open sets
 $U_i$ ($i=0, 1, 2, 3$) in $U$ such that $f_{0}^{-1}(U_0)$
 is also disjoint from $U_i$ ($i=1,\ldots, 3$)
 and $V_i$ ($i=1,\ldots, 5$).
 It is not difficult to choose such $U_i$ and $V_i$, so we leave
 it to the reader.
Finally, we take an $(f_1,f_2)$-generic point $\bar{p}$ in $U_0$
such that $f^{-1}_0(\bar{p})$ is also $(f_1,f_2)$-generic.
One can check that such a choice is indeed possible
by using the fact that being an $(f_1,f_2)$-generic point
is a generic property in $J$.

Now we construct the periodic orbit by using our connecting lemmas.
First, by Lemma \ref{lemma:connecting 2},
with $\bar{p}, p_4, p_1, U_1$ and $V_1$ corresponding to
$p, p' ,\hat{p}, U$ and $V$ respectively and
 $F$ being the identity germ,
 we obtain $\omega_1 \in \{0,1,2\}^*$ and $h_1 \in \cN^{1/4}$ such that
 $\supp(h_1) \subset (U_1 \cup V_1)$,
 $\rho_{h_1}^{\omega_1}(\bar{p})=p_4$,
 and $H_1=[\rho_{h_1}^{\omega_1}]_{\bar{p}}$ is $1$-flat.
Similarly, we apply Lemma \ref{lemma:connecting 2}
to $(p_4, p_5, p_2, U_2, V_2)$ and, again, letting $F$ be the identity germ,
and to $(p_5, f_0^{-1}(\bar{p}), p_3, U_3, V_3)$ and
$F=([f_0]_{f_0^{-1}(\bar{p})})^{-1}$, and obtain that
 there exist $\omega_2,\omega_3 \in \{0,1,2\}^*$
 and $h_2,h_3 \in \cN^{1/4}$ such that
\begin{enumerate}
 \item $\supp(h_i) \subset (U_i \cup V_i) $ for $i=2,3$,
 \item $\rho_{h_2}^{\omega_2}(p_4)=p_5$,
 $H_2=[\rho_{h_2}^{\omega_2}]_{p_4}$ is $1$-flat, and
 \item $\rho_{h_3}^{\omega_3}(p_5)=f_0^{-1}(\bar{p})$,
 $H_3=[f_0]_{f_0^{-1}(\bar{p})} \circ [\rho_{h_3}^{\omega_3}]_{p_5}$
 is $1$-flat.
\end{enumerate}

Let $h_{\flat} = h_1\circ h_2 \circ h_3$ and
$\gamma_{m,n}=\omega_3 (0 \gamma_5)^m \omega_2 (0 \gamma_4)^n \omega_1$ for $m,n \geq 1$.
Notice that, by construction, $\bar{p}$ is
a $1$-flat $0\gamma_{m, n}$-periodic point of $\rho_{h_{\flat}}$.
We show that by adding a further perturbation and choosing
$m, n$ appropriately, we can obtain the $2$-flatness and
the condition on the Schwarzian derivative.

Let $\alpha=\sum_{i=1}^3 A(H_i)$ and $\beta=\sum_{i=1}^3 S(H_i)$.
By Lemma \ref{lemma:2-flat}
(where we let $F_1 = [\rho^{0\gamma_4}]_{p_4}$
and $F_2 = [\rho^{0\gamma_5}]_{p_5}$),
together with Remark~\ref{rmk:realization},
we take $h_4 \in \cN^{1/4}$ and $m,n \geq 1$ such that
 $\supp(h_4) \subset V_4$, $h_4(p_4)=p_4$, $[h_4]_{p_4}$ is $1$-flat, and the following holds:
\begin{gather*}
 A([\rho^{0 \gamma_5}]_{p_5}^m)
 + A(([h_4]_{p_4} \circ [\rho^{0 \gamma_4}]_{p_4})^n)+\alpha=0.\\
 S([\rho^{0 \gamma_4}]_{p_4})
 \cdot \left\{
 S([\rho^{0 \gamma_5}]_{p_5}^m)
 +S(([h_4]_{p_4} \circ [\rho^{0 \gamma_4}]_{p_4})^n)+\beta \right\} >0.
\end{gather*}

Let $h=h_1 \circ h_2 \circ h_3 \circ h_4 =h_{\flat}\circ h_4$.
This is a map in $\cN$ such that
 $\supp(h) \subset (\bigcup_{i=1}^3 U_i \cup \bigcup_{i=1}^4 V_i)$.
Since $U_1, U_2, U_3, V_1,\dots,V_5$ are mutually
 disjoint and they do not intersects with
 $\{\bar{p}, f_0^{-1}(\bar{p})\} \cup \bigcup_{i=1}^5 \Sigma^{\gamma_i}(p_i)$,
 we have
\begin{align*}
[\rho_h^{0\gamma_{m,n}}]_{\bar{p}}
 & = [\rho_{h_3}^{0\omega_3}]_{p_5}
  \circ [\rho^{0\gamma_5}]_{p_5}^m
  \circ [\rho_{h_2}^{\omega_2}]_{p_4}
  \circ [\rho_{h_4}^{0\gamma_4}]_{p_4}^n
  \circ [\rho_{h_1}^{\omega_1}]_{\hat{p}}\\
 & = H_3 \circ [\rho^{0\gamma_5}]_{p_5}^m
 \circ H_2 \circ [\rho_{h_4}^{0\gamma_4}]_{p_4}^n \circ H_1.
\end{align*}
This, together with the $1$-flatness of each germ, implies that
\begin{align*}
A([\rho_h^{0 \gamma_{m,n}}]_{\bar{p}})
 & = A([\rho^{0 \gamma_5}]_{p_5}^m)
 + A(([h_4]_{p_4} \circ [\rho^{0 \gamma_4}]_{p_4})^n)+\alpha,\\
S([\rho_h^{0 \gamma_{m,n}}]_{\bar{p}})
 & = S([\rho^{0 \gamma_5}]_{p_5}^m)
 + S(([h_4]_{p_4} \circ [\rho^{0 \gamma_4}]_{p_4})^n)+\beta.
\end{align*}
Therefore,
$A([\rho_h^{0 \gamma_{m,n}}]_{\bar{p}})=0$
 and
$S([\rho_h^{0 \gamma_{m,n}}]_{\bar{p}})
\cdot S([\rho^{0\gamma_4}]_{p_4}) >0$.
By construction, one can see that
  $\Sigma_h^{\gamma_{m,n}}(\bar{p})
 \subset
\left( \bigcup_{i=1}^3 U_i \cup \bigcup_{i=1}^5 (\Sigma^{\gamma_i}(p_i) \cup V_i ) \right)$.
In particular,
 $\bar{p} \not\in \Sigma_h^{\gamma_{m,n}}(\bar{p})$
 and $\Sigma_h^{\gamma_{m,n}}(\bar{p})
  \subset \bigcup_{i=1}^5 \Sigma^{\gamma_i}(p_i) \cup U \cup V$.

Thus we have constructed the desired
$\hat{p} = \bar{p}$, $\hat{\gamma}=\gamma_{m, n}$ and $h$.
\end{proof}

The following lemma is the third step.
\begin{lemma}
\label{lemma:3-flat periodic}
Suppose that
 there exist mutually distinct $(f_1,f_2)$-generic points
 $p_1,\dots,p_5$ in $J$
 and $\gamma_1,\dots, \gamma_5 \in \{0,1,2\}^*$ such that
\begin{enumerate}
 \item each $p_i$ is a $2$-flat $0 \gamma_i$-periodic point,
 \item $\Sigma^{\gamma_1}(p_1), \dots, \Sigma^{\gamma_5}(p_5)$
 and $\{p_1,\dots,p_5\}$ are mutually disjoint,
 \item $S([\rho^{0 \gamma_4}]_{p_4})
 \cdot S([\rho^{0 \gamma_5}]_{p_5})<0$.
\end{enumerate}
Then, for any neighborhood $V$ of $\{p_1,\dots,p_5\}$,
any neighborhood $\cN$ of the identity map in $\Diff^\infty([0,1])$,
 and any non-empty open subset $U$ of $J \cap f_0(J)$,
 there exist an $(f_1,f_2)$-generic point $\hat{p} \in U$,
 $\gamma \in \{0,1,2\}^*$, and $h \in \cN$ such that
 $\supp(h) \subset U \cup V$,
 $\hat{p}$ is a $3$-flat $0 \gamma$-periodic point for $\rho_h$
 with $\hat{p} \not\in \Sigma^\gamma(\hat{p})$, and
 $\Sigma^\gamma(\hat{p})
 \subset \bigcup_{i=1}^5 \Sigma^{\gamma_i}(p_i) \cup U \cup V$.
\end{lemma}
\begin{proof}
The proof is done similarly to the proof of Lemma \ref{lemma:2-flat periodic},
with the use of Lemma \ref{lemma:(s+1)-flat} instead of Lemma \ref{lemma:2-flat}.

Without loss of generality,
 we may assume that $U$, $V$ and $\bigcup_{i=1}^5\Sigma^{\gamma_i}(p_i)$
 are mutually disjoint.
Take a neighborhood $\cN^{1/4}$ of the identity in $\Diff^\infty([0,1])$
 such that any composition of four diffeomorphisms
 in $\cN^{1/4}$ belongs to $\cN$.
We fix open neighborhoods $V_i$ of $p_i$ for each $i=1,\dots,5$
 such that $V_1,\dots,V_5$ are mutually disjoint subsets of $V$.
We also take mutually disjoint non-empty open sets
$U_i$ ($i =0, 1, 2, 3$), such that $f_0^{-1}(U_0)$
is also disjoint from $U_i$ ($i=0, 1, 2, 3$) and
 $V_i$ ($i = 1, \ldots 5$).
Finally, we take an $(f_1, f_2)$-generic point
$\bar{p} \in U_0$ such that
$f_0^{-1}(\bar{p})$ is also $(f_1, f_2)$-generic.

Then, by applying Lemma \ref{lemma:connecting 2} with $r_0=2$ in
 the same way as Lemma \ref{lemma:2-flat periodic},
 we obtain $\omega_1,\omega_2,\omega_3 \in \{0,1,2\}^*$,
 $h_1,h_2,h_3 \in \cN^{1/4}$ such that
\begin{enumerate}
 \item  $\supp(h_i) \subset (U_i \cup V_i)$
 for $i=1,2,3$,
 \item $\rho_{h_1}^{\omega_1}(\bar{p})=p_4$,
  $H_1=[\rho_{h_1}^{\omega_1}]_{\bar{p}}$ is $2$-flat,
 \item $\rho_{h_2}^{\omega_2}(p_4)=p_5$,
 $H_2=[\rho_{h_2}^{\omega_2}]_{p_2}$ is $2$-flat, and
 \item $\rho_{h_3}^{\omega_3}(p_5)=f_0^{-1}(\bar{p})$,
 $H_3=[f_0]_{f_0^{-1}(\bar{p})} \circ [\rho_{h_3}^{\omega_3}]_{p_2}$
 is $2$-flat.
\end{enumerate}
Since each $H_i$ is $2$-flat,
 there exists $\alpha_0 \in \RR$ such that
\begin{equation*}
 H_3 \circ H_2 \circ H_1(x)=x+\alpha_0 x^3 +o(x^3).
\end{equation*}
By Lemma~\ref{lemma:(s+1)-flat} for $r=2$
and $\alpha = \alpha_0$
(notice that since $p_i$ is $2$-flat, our third assumption implies
$[\rho^{0\gamma_4}]_{p_4}^{(3)} \cdot
[\rho^{0\gamma_5}]_{p_5}^{(3)}<0$),
 there exist $h_4 \in \cN^{1/4}$ and $m,n \geq 1$ such that
 $\supp(h_4) \subset V_4$, $h_4(p_4)=p_4$, $[h_4]_{p_4}$ is $2$-flat, and
\begin{equation*}
[\rho^{0\gamma_5}]_{p_5}^m \circ
 ([h_4]_{p_4} \circ [\rho^{0\gamma_4}]_{p_4})^n(x)
 =x-\alpha_0 x^3+o(x^3).
\end{equation*}
Let $h=h_4 \circ h_3 \circ h_2 \circ h_1$.
This is a map in $\cN$ such that
 $\supp(h) \subset (\bigcup_{i=1}^3 U_i \cup \bigcup_{i=1}^4 V_i)$.
We also put
 $\gamma_{m,n}=\omega_3 (0 \gamma_5)^m \omega_2 (0\gamma_4)^n \omega_1$.
Since $V_1,\dots,V_5$ are mutually
 disjoint and they do not intersects with
 $\{\bar{p}, f_{0}^{-1}(\bar{p})\} \cup \bigcup_{i=1}^5 \Sigma^{\gamma_i}(p_i)$,
 we have
\begin{align*}
[\rho_h^{0\gamma_{m,n}}]_{\bar{p}}
 & = [\rho_{h_3}^{0\omega_3}]_{p_5}
  \circ [\rho^{0\gamma_5}]_{p_5}^m
  \circ [\rho_{h_2}^{\omega_2}]_{p_4}
  \circ [\rho_{h_4}^{0\gamma_4}]_{p_4}^n
  \circ [\rho_{h_1}^{\omega_1}]_{\bar{p}}\\
 & = H_3 \circ [\rho^{0\gamma_5}]_{p_5}^m
  \circ H_2 \circ ([h]_{p_4} \circ [\rho^{0\gamma_4}]_{p_4})^n \circ H_1.
\end{align*}
Since any $2$-flat germs commute with each other modulo $o(x^3)$,
 this implies that
 the germ $[\rho_h^{0\gamma_{m,n}}]_{\bar{p}}$ is $3$-flat.
As in the proof of Lemma~\ref{lemma:2-flat periodic},
by construction we can check
  $\bar{p} \not\in \Sigma_h^{\gamma_{m,n}}(\bar{p})$
 and $\Sigma_h^{\gamma_{m,n}}(\bar{p})
  \subset \bigcup_{i=1}^5 \Sigma^{\gamma_i}(p_i) \cup U \cup V$.

Thus $\hat{p} =\bar{p}$ is the desired 3-flat
$(0\gamma_{m, n})$-periodic point for $\rho_{h}$.
\end{proof}

The fourth step is the following
\begin{lemma}
\label{lemma:s>3-flat periodic}
Suppose that $3 \leq r <\infty$
 and there exist mutually distinct $(f_1,f_2)$-generic points
 $p_1,\dots,p_9$ in $J$
 and $\gamma_1,\dots, \gamma_9 \in \{0,1,2\}^*$ such that
\begin{enumerate}
 \item each $p_i$ is an $r$-flat $0 \gamma_i$-periodic point, and
 \item $\Sigma^{\gamma_1}(p_1), \dots, \Sigma^{\gamma_9}(p_9)$,
 and $\{p_1,\dots,p_9\}$ are mutually disjoint.
\end{enumerate}
Then, for any neighborhood $V$ of $\{p_1,\dots,p_9\}$,
any neighborhood $\cN$ of the identity map in $\Diff^\infty([0,1])$,
 and any non-empty open subset $U$ of $J \cap f_0(J)$,
 there exists an $(f_1,f_2)$-generic point $\hat{p} \in U$,
 $\gamma \in \{0,1,2\}^*$, and $h \in \cN$ such that
 $\supp(h) \subset U \cup V$,
 $\hat{p}$ is an $(r+1)$-flat $0 \gamma$-periodic point for $\rho_h$
 with $\hat{p} \not\in \Sigma^\gamma(\hat{p})$,
 $\Sigma^\gamma(\hat{p})
 \subset \bigcup_{i=1}^9 \Sigma^{\gamma_i}(p_i) \cup U \cup V$.
\end{lemma}
\begin{proof}
One proves this in the same way as Lemma \ref{lemma:3-flat periodic},
 using Lemma \ref{lemma:s>3-flat} instead of Lemma \ref{lemma:(s+1)-flat}. Hence we omit the details.
\end{proof}

%
%

Now we finish the proof of Proposition \ref{prop:flat periodic points}.
The case $r=\infty$ is reduced to the case $r < \infty$.
So, first we consider the case $r<\infty$.

\begin{proof}
[Proof of Proposition \ref{prop:flat periodic points} for $r < \infty$]
Take $\rho=(f_0,f_1,f_2) \in \cW^\infty$.
Let $J$ be the closed subinterval of $[0,1]$ on which $(f_1,f_2)$
 is a blender,
 $(\bar{p}_1,\bar{q}_1), \dots, (\bar{p}_4,\bar{q}_4)$
  be repeller-attractor pairs in $\Int J$,
 $z_i \in W^u(\bar{p}_i) \cap W^s(\bar{q}_i)$ $(i=1,2,3,4)$
 be heteroclinic points such that
 $\tau_A(z_1,f_0)>0>\tau_A(z_2,f_0)$ 
 and $\tau_S(z_3,f_0)>0>\tau_S(z_4,f_0)$.
By perturbing $f_0$ if it is necessary,
 we may assume that $\bar{p}_i$ is $(f_1,f_2)$-generic,
 $(\log f_0'(\bar{p}_i))/(\log f_0'(\bar{q}_i))$ are irrational,
 $\tau_A(z_i,f_0)\neq 0$, and $\tau_S(z_i,f_0) \neq 0$ for $i=1,\dots,4$.
By an elementary combinatorial argument\footnote{
If $\tau_S(z_1,f_0)\neq \tau_S(z_2)$, then choose $(i_1,i_2)=(1,2)$.
If $\tau_S(z_1,f_0)= \tau_S(z_2)=+1$ and $\tau_A(z_4,f_0)=+1$,
 then choose $(i_1,i_2)=(2,4)$. Other cases are similar.},
 there exists a pair $(i_1,i_2)$ in $\{1,2,3,4\}$ such that
 the signs of $z_{i_1}$ and $z_{i_2}$ are opposite.

First, we show how
we construct a $3$-flat periodic point.
Fix a neighborhood $\cN$ of the identity map
 in $\Diff^{\infty}([0,1])$.
Take a smaller neighborhood $\cN^{1/3}$ of the identity
 such that any composition of three maps in $\cN^{1/3}$ belongs to $\cN$.
Applying Lemma \ref{lemma:connecting fl} repeatedly
 to the heteroclinic points $z_{i_1}$ and $z_{i_2}$,
 we obtain a diffeomorphism $h_1 \in \cN^{1/3}$,
 $(f_1,f_2)$-generic points $p_1,\dots,p_{25} \in J$,
 and $\gamma_1,\dots,\gamma_{25}\in \{0, 1, 2\}^{\ast}$ such that
 $p_i$ is a $1$-flat $0\gamma_i$-periodic point
 satisfying the following:
\begin{itemize}
\item $\Sigma_{h_1}^{\gamma_1}(p_1),\dots,\Sigma_{h_1}^{\gamma_{25}}(p_{25})$,
 and $\{p_1,\dots,p_{25}\}$ are mutually disjoint,
\item
 $[\rho_{h_1}^{0 \gamma_{5j-1}}]_{p_{19}}$
 and $[\rho_{h_1}^{0 \gamma_{5j}}]_{p_{24}}$
 have opposite signs and
 $|(S/A)([\rho_{h_1}^{0 \gamma_{5j-1}}]_{p_{5j-1}})|
 >|(S/A)([\rho_{h_1}^{0 \gamma_{5j}}]_{p_{5j}})|$
 for any $j=1,\dots,5$,
\item
 $S([\rho_{h_1}^{0 \gamma_{19}}]_{p_{19}})>0>
 S([\rho_{h_1}^{0 \gamma_{24}}]_{p_{24}})$.
\end{itemize}
Notice that we may apply Lemma \ref{lemma:connecting fl}
repeatedly so that the perturbations do not interfere with each other,
since Lemma \ref{lemma:connecting fl} allows us to localize
the support of perturbations away from a given finite set.

Then, by applying Lemma~\ref{lemma:2-flat periodic}
 to each quintuple of $1$-flat periodic points
 $(p_{5j-4},\dots,p_{5j})$ of $\rho_{h_1}$ ($j=1,\ldots, 5$),
 we take $h_2 \in \cN^{1/3}$,
 $p'_j \in J$, and $\gamma'_j \in \{0, 1, 2\}^{\ast}$
  such that
 $p'_j$ is an $(f_1, f_2)$-generic, $2$-flat
 $0\gamma'_j$-periodic point of $\rho_{h_2 \circ h_1}$
 for $j=1,\ldots 5$,
 $\Sigma_{h_2 \circ h_1}^{\gamma'_1}(p'_1),
   \dots,\Sigma_{h_2 \circ h_1}^{\gamma'_5}(p'_5)$,
 and $\{p'_1,\dots,p'_5\}$ are mutually disjoint, and
\begin{equation*}
S([\rho_{h_2 \circ h_1}^{0 \gamma'_4}]_{p'_4})>0>
 S([\rho_{h_2 \circ h_1}^{0 \gamma'_5}]_{p'_5}).
\end{equation*}

Then, by applying Lemma \ref{lemma:3-flat periodic} to $(p'_1,\dots,p'_5)$,
 we take a diffeomorphism $h_3 \in \cN^{1/3}$,
 an $(f_1,f_2)$-generic point $\hat{p}$ in $J$,
 and $\hat{\gamma} \in \{0,1,2\}^*$
 such that
 $\hat{p}$ is a $3$-flat $0\hat{\gamma}$-periodic point of
 $\rho_{h_3 \circ h_2 \circ h_1}$
 and $\hat{p} \not\in \Sigma_h^{\hat{\gamma}}(\hat{p})$.
Thus we have constructed a $3$-flat
$0\gamma$-periodic point $\hat{p}$ of $\rho_{h}$,
where $h=h_3 \circ h_2 \circ h_1$ is a diffeomorphism in $\cN$.
Notice that the length of $\gamma$ can be taken arbitrarily large,
since by Remark~\ref{rmk:length} we can assume that the
lengths of the $1$-flat periodic points produced in the first step
are arbitrarily large.

In a similar way, for $r \geq 3$, by Lemma \ref{lemma:s>3-flat periodic}
 we can construct an $(r+1)$-flat periodic point
 with an arbitrary large period from nine $r$-flat periodic points
 by a small perturbation.
Hence, we obtain $N$ of $r$-flat periodic points
 with an arbitrary large period
 starting with $25\cdot 9^{r-3}N$ of $1$-flat periodic points.
\end{proof}

Finally, let us consider the case $r = \infty$.
\begin{proof}
[Proof of Proposition \ref{prop:flat periodic points} for $r = \infty$]
For simplicity we only consider the case $N=1$ (the proof for the
general case is done similarly).
First, as in the proof of Lemma~\ref{lemma:connecting inf},
given a neighborhood
$\cN \subset \mathrm{Diff}^{\infty}([0, 1])$ of the identity map,
there exist $s >1$ and
a $C^s$-neighborhood $\cN^s \subset \mathrm{Diff}^{s}([0, 1])$ of the
identity map such that
$\mathcal{N}^s \cap \mathrm{Diff}^{\infty}([0, 1]) \subset
\cN$.

For $\rho \in \cW^{\infty}$ in the assumption of the proposition,
we apply the already proven result for finite $r$ with $r =s$.
This gives us $h \in \cN$,
$\hat{p} \in J$, and $\gamma \in \{0, 1, 2\}^{\ast}$
such that $\hat{p}$ is an $s$-flat $0\gamma$-periodic point
for $\rho_{h}$ such that $\hat{p} \not\in \Sigma^{\gamma}(\hat{p})$.
Now, by Remark~\ref{rem:flat}, we take a
$C^{\infty}$-diffeomorphism $g$
which is $C^s$-close to the identity
such that $[g]_{\hat{p}} \circ [\rho^{0\gamma}_h]_{\hat{p}}$
is the identity germ and
$\mathrm{supp}(g) \cap \Sigma^{\gamma}_h(\hat{p}) = \emptyset$.
By choosing $g$ sufficiently $C^s$-close to the identity,
we can assume that $g \circ h \in \cN^s$.
Accordingly, we have $g \circ h \in \cN$. Now,
$\hat{p}$ is an $\infty$-flat $0\gamma$-periodic point for $\rho_{g \circ h}$.
Thus the proof is completed.

\end{proof}


\section{Universal semigroups}\label{sec:universal}

In this section, we prove the following proposition
which,  together with Proposition~\ref{prop:flat periodic points},
implies Theorem \ref{thm:universal}.
 \begin{prop}
\label{prop:universal}
Let $\rho=(f_0,f_1,f_2)$ be an element of
$\cA^{\infty}(\{0,1,2\})$
 and $(\theta_1,\dots,\theta_N)$
 be in $(\mathcal{E}^{\infty})^N$.
Suppose that there exist distinct $(f_1,f_2)$-generic points
 $p_1,\dots,p_{4N}$ and words $\gamma_1,\dots,\gamma_{4N} \in \{0,1,2\}^*$
 such that $p_i$ is an $\infty$-flat $0 \gamma_i$-periodic point
 for each $i=1,\dots,4N$,
 and the sets $\{p_1,\dots,p_{4N}\}$,
 $\Sigma^{\gamma_1}(p_1)$, \dots, $\Sigma^{\gamma_{4N}}(p_{4N})$
 are mutually disjoint.
Then, for any neighborhood $\cN$ of the identity in $\Diff^{\infty}([0,1])$,
 there exist a map $h \in \cN$,
 a closed interval $I \subset [0,1]$,
 an affine diffeomorphism $\Phi:[0,1] \ra I$,
 and words $\omega_1,\dots,\omega_N \in \{0,1,2\}^*$
 such that $\Phi \circ \theta_k = \rho_{h}^{\omega_k} \circ \Phi$
for every $k=1,\dots,N$.
\end{prop}
First, let us derive
Theorem \ref{thm:universal} from Proposition \ref{prop:universal}.
\begin{proof}
[Proof of Theorem \ref{thm:universal} from Proposition \ref{prop:universal}]
Fix $1 \leq r \leq \infty$.
The space $(\cE^r)^N$ admits a countable open basis $(\cO_n)_{n=1}^\infty$.
Let $\cU_n$ be the set consisting of $\rho \in \cW^r$
 which realize the semigroup action generated by some element of $\cO_n$.
The set $\cU_n$ is an open subset of $\cW^r$ and any element in $\bigcap_{n \geq 1}\cU_n$
generates a universal semigroup. Hence, it is sufficient to show that
every $\cU_n$ is a dense subset of $\cW^r$.

To see this, we fix a non-empty open subset $\cU$ of $\cW^r$ and an element $(\theta_1,\dots,\theta_N)$ of
$\cO_n \cap (\cE^{\infty})^N$ (notice that $\cO_n \cap (\cE^{\infty})^N$ is dense in $\cO_n$).
By Proposition~\ref{prop:flat periodic points} and the density of $\cW^\infty$ in $\cW^r$,
we take $\tilde{\rho} \in \cU \cap \cA^\infty(\{0,1,2\})$ which satisfies the hypothesis of Proposition~\ref{prop:universal}.
Then, by Proposition~\ref{prop:universal} there exist $\rho=(f_0,f_1,f_2) \in \cU \cap \cA^\infty(\{0,1,2\})$,
 a closed interval $I \subset [0,1]$,
 an affine diffeomorphism $\Phi:[0,1] \ra I$,
 and words $\omega_1,\dots,\omega_N \in \{0,1,2\}^*$
 such that
 $\Phi \circ \theta_k = (\rho^{\omega_k}|_I) \circ \Phi$
for every $k=1,\dots,N$.
This implies that $\rho$ realizes the semigroup action
 generated by $(\theta_1,\dots,\theta_N)$.
Therefore $\cU$ intersects with $\cU_n$.
Since the choice of $\cU$ is arbitrary,
 $\cU_n$ is a dense subset of $\cW^r$.
\end{proof}

Now we prove Proposition \ref{prop:universal}.

\begin{proof}
[Proof of Proposition \ref{prop:universal}]
Let $\cN^\#$ be a neighborhood of the identity in $\Diff^\infty([0,1])$
 such that any composition of $6N$ maps in $\cN^\#$ belongs to $\cN$.
Take a neighborhood $V_i$ of $p_i$ for each $i=1,\dots, 4N$,
 non-empty open subsets
 $U_1,\dots,U_{4N}$, $W_1,\dots,W_{2N}$ of $J \cap f(J)$,
 and an $(f_1,f_2)$-generic point $\bar{p} \in J$
 such that the sets
\begin{gather*}
 \{\bar{p}\},\; V_1, \dots, V_{4N},\; U_1,\dots,U_{4N},\; W_1,\dots, W_{2N},\\
  f_0^{-1}(W_1),\dots, f_0^{-1}(W_{2N}),
 \; \bigcup_{i=1}^{4N}\Sigma^{\gamma_i}(p_i)  
\end{gather*}
 are mutually disjoint.
 It is not difficult to check
 that such choice is possible (the detail is left to the reader).
Then choose $q_i \in W_i$ ($i=1,\ldots, 2N$)
so that $q_i$
 and $f_0^{-1}(q_i)$ are $(f_1,f_2)$-generic.
For each $j=1,\dots,2N$, by applying Lemma~\ref{lemma:connecting inf}
(for the neighborhood $V_j$, the open subset $U_j$, and the triple $(\bar{p}, f_0^{-1}(q_j) , p_j)$ taken as $(p, p', \hat{p})$)
we obtain $h_j \in \cN^\#$ and $\eta_j \in \{0,1,2\}^*$ such that
\begin{itemize}
 \item $\supp(h_j) \subset U_j \cup V_j$,
  $\Sigma_{h_j}^{\eta_j}(\bar{p})
 \subset \Sigma^{\gamma_j}(p_j) \cup U_j \cup  V_j$,
 \item $\rho_{h_j}^{\eta_j}(\bar{p})=f_0^{-1}(q_j)$, and
 $[f_0]_{f_0^{-1}(q_j)} \circ [\rho_{h_j}^{\eta_j}]_{\bar{p}}$
 is $\infty$-flat.
\end{itemize}
Similarly, for each $j=1,\dots,2N$,
by viewing $(q_j, \bar{p}, p_{2N+j})$ as $(p, p', \hat{p})$
in Lemma~\ref{lemma:connecting inf},
we take
 $h_{2N+j} \in \cN^\#$ and $\eta_{2N+j} \in \{0,1,2\}^*$
 (for $j = 1,\ldots, 2N$) such that
\begin{itemize}
 \item $\supp(h_{2N+j}) \subset U_{2N+j} \cup V_{2N+j}$,
 $\Sigma_{h_{2N+j}}^{\eta_{2N+j}}(q_j)
 \subset \Sigma^{\gamma_{2N+j}}(p_{2N+j}) \cup U_{2N+j} \cup V_{2N+j}$,
 \item $\rho_{h_{2N+j}}^{\eta_{2N+j}}(q_j)=\bar{p}$, and
 $[\rho_{h_{2N+j}}^{\eta_{2N+j}}]_{q_j}$ is $\infty$-flat.
\end{itemize}
Put $\bar{\omega}_j=\eta_{2N+j} 0 \eta_j$ for $j=1,\dots,2N$.
Since the germ
 $[\rho_{h_{2N+j}}^{\eta_{2N+j}}]_{q_j}
 \circ [\rho_{h_j}^{0 \eta_j}]_{\bar{p}}$ is $\infty$-flat
 and $q_j \not\in \Sigma_{h_{2N+j}}^{\eta_{2N+j}}(q_j) \cup
 \Sigma_{h_j}^{\eta_j}(\bar{p})$,
 there exists $\bar{h}_j \in \cN^\#$ such that
 $\supp{\bar{h}_j} \subset W_j$, $\bar{h}_j(q_j)=q_j$,
 and
\begin{equation*}
[\rho_{\hat{h}_j \circ h_{2N+j}
\circ h_j}^{\bar{\omega}_j}]_{\bar{p}}
 = [\rho_{h_{2N+j}}^{\eta_{2N+j}}]_{q_j}
  \circ [\bar{h}_j]_{q_j} \circ
  [\rho_{h_j}^{0 \eta_j}]_{\bar{p}}
\end{equation*}
 is equal to the identity map as a germ.
In particular,
 there exists an open interval $I_0$ containing $\bar{p}$ such that
 $\rho_{\bar{h}_j \circ h_{2N+j} \circ h_j}^{\bar{\omega}_j}(x)=x$
 for all $x \in I_0$.
Put $h_\#=h_1 \circ \dots \circ h_{4N}
\circ \bar{h}_1 \circ \dots \circ \bar{h}_{2N}$.
Remark that $h_\#$ is a composition of $6N$ maps in $\cN^\#$,
and hence, it belongs to $\cN$.
Now, $\rho_{h_\#}^{\bar{\omega}_j}$ is the identity map on $I_0$.
Notice that this is also true for smaller intervals in $I_0$
containing $\bar{p}$.
We shrink $I_0$ so that $\rho^{0\eta_j}_{h_\#}(I_0)\subset W_j$
for every $j=1,\ldots, 2N$.

For maps $(\theta_1, \ldots, \theta_N)$
in the assumption of Proposition~\ref{prop:universal},
we take their extension
$(\bar{\theta}_1, \ldots, \bar{\theta}_N)$ over $\mathbb{R}$
in such a way that each $\bar{\theta}_i$ has compact support.
Now, for each $\bar{\theta}_i$
we apply Lemma \ref{lemma:composition}
for $I = [0, 1]$ to obtain
one-parameter groups
$(\tilde{G}_1^t)_{t \in \RR}, \dots, (\tilde{G}_N^t)_{t \in \RR}$,
$(\tilde{H}_1^t)_{t \in \RR}, \dots, (\tilde{H}_N^t)_{t \in \RR}$ of
$C^{\infty}$-diffeomorphisms of $\mathbb{R}$ having compact support and
satisfying $\tilde{G}^1_i \circ \tilde{H}^1_i|_{[0, 1]} = \theta_i$ for every $i$.
Let $I_1$ be a compact interval in $\mathbb{R}$ which contains
$\mathrm{supp}(\tilde{G}_i^t) $ and $\mathrm{supp}(\tilde{H}_i^t) $
for every $i$.

Now we take an affine map $\Phi:\RR \ra \RR$ satisfying
$\Phi(I_1) = I_0$ and define
 one-parameter groups $(G_i^t)$ and $(H_i^t)$\, ($i=1,\ldots, N$)
 of $\mathrm{Diff}^{\infty}(\mathbb{R})$
 by $G_i^t = \Phi \circ \tilde{G}_i^t \circ \Phi^{-1}$ and
$H_i^t = \Phi \circ \tilde{H}_i^t \circ \Phi^{-1}$.
Notice that $\mathrm{supp}(G_i^t) $ and
$\mathrm{supp}(H_i^t) $ are contained in $I_0$,
hence their restriction on $[0, 1]$ give diffeomorphisms of $[0, 1]$.
Also, by definition we have
 $(G_i^1 \circ H_i^1) \circ \Phi =\Phi \circ \bar{\theta}_i
 = \Phi \circ \theta_i $
 on $[0,1]$ for every $i = 1, \ldots, N$.
Take a continuous family
 $(\bar{h}^t)_{t \in \RR}$ of diffeomorphism of $[0,1]$
defined as follows: for $k=1,\dots,N$, we put
\begin{equation*}
\bar{h}^t(x)=
\begin{cases}
\, \rho_{h_\#}^{0 \eta_{2k-1}} \circ H_k^t
 \circ (\rho_{h_\#}^{0 \eta_{2k-1}})^{-1}(x)
 & (x \in W_{2k-1}),\medskip\\
\, \rho_{h_\#}^{0 \eta_{2k}} \circ G_k^t
 \circ (\rho_{h_\#}^{0 \eta_{2k}})^{-1}(x)
 & (x \in W_{2k}),
\end{cases}
\end{equation*}
and then extend them as the identity map outside $W_j$.
Notice that this is a well-defined procedure
because of the choice of $I_0$.
Then, we have
\begin{align*}
\rho_{\bar{h}^t \circ h_\#}^{\bar{\omega}_{2k-1}}& =
 \rho_{h\#}^{\eta_{2N+2k-1}} \circ \bar{h}^t
 \circ \rho_{h_\#}^{0  \eta_{2k-1}} =H_k^t,\\
\rho_{\bar{h}^t \circ h_\#}^{\bar{\omega}_{2k}}& =
 \rho_{h_\#}^{\eta_{2N+2k}} \circ \bar{h}^t
 \circ \rho_{h_\#}^{0  \eta_{2k}} = G_k^t
\end{align*}
 on $I_0$.
Hence, for any $m \geq 1$,
\begin{equation*}
 \rho_{\bar{h}^{1/m} \circ h_\#}^{\bar{\omega}_{2k}^m\bar{\omega}_{2k-1}^m}
 = G_k^1 \circ H_k^1= \Phi \circ \theta_k \circ \Phi^{-1}
\end{equation*}
 on $I_0$.
If $m$ is sufficiently large,
 $\bar{h}^{1/m} \circ h_\#$ is contained in $\cN$.
Therefore, $h=\bar{h}^{1/m} \circ h_\#$,
 $(\omega_k=\bar{\omega}_{2k}^m\bar{\omega}_{2k-1}^m)_{k=1}^N$,
 and the map $\Phi$ satisfy the statement of the Proposition for sufficiently large $m$.
\end{proof}

%
%

\section{Wild behavior along generic infinite words} \label{sec:infinite path}
In this section we prove Theorem \ref{thm:growth along path}.
We start with a general lemma on generic infinite words.
\begin{lemma}
\label{lemma:generic path}
Let $X$ be a Baire space, $k$ be a positive integer,
 and $(X(\omega))_{\omega \in \{1,\dots,k\}^*}$ be a family
 of open subsets of $X$.
Suppose that $\bigcup_{\eta \in \{1,\dots,k \}^*}X(\eta\omega)$
 is dense in $X$ for any $\omega \in \{1,\dots,k\}^*$.
Then, for generic $x \in X$, the set
\begin{equation*}
 \{\,\ub{\omega} \in \{1,\dots,k\}^\infty \mid
 x \in X(\ub{\omega}|_n) \text{\rm\; for infinitely many $n$'s}\}
\end{equation*}
 is a residual subset of $\{1,\dots,k\}^\infty$.
\end{lemma}
\begin{proof}
By the assumption, the set
\begin{equation*}
 \bigcap_{\omega \in \{1,\dots,k\}^*}
 \left(\bigcup_{\eta \in \{1,\dots,k\}^*} X(\eta\omega) \right),
\end{equation*}
 is a residual subset of $X$.
Fix a point $x$ in this subset.
For $n \geq 1$, put
\begin{equation*}
 \Omega_n (x)=\{\ub{\omega} \in \{1,\dots,k\}^\infty \mid x \in X(\ub{\omega}|_n)\}.
\end{equation*}
This is an open subset of $\{1,\dots,k\}^\infty$.
By the choice of $x$,
 for each $\ub{\omega} \in \{1,\dots,k\}^\infty$ and $m \geq 1$,
 there exists $\eta_m \in \{1,\dots,k\}^*$ such that
 $x \in X(\eta_m(\ub{\omega}|_m))$.
Hence,
 the open set
 $\{\ub{\omega}' \in \{1,\dots,k\}^* \mid \ub{\omega}'|_m=\ub{\omega}|_m\}$
 intersects with $\Omega_{m+|\eta_m|}(x)$ for any $m$.
This implies that $\bigcup_{n \geq N}\Omega_n (x)$ is
 a dense subset of $\{1,\dots,k\}^\infty$ for any $N \geq 1$.
Hence, $\bigcap_{N \geq 1}\bigcup_{n \geq N}\Omega_n (x)$
 is a residual subset of $\{1,\dots,k\}^\infty$.
Any infinite word $\ub{\omega}$ in this residual subset satisfies
 $x \in X(\ub{\omega}|_n)$ for infinitely many $n$'s.
\end{proof}
Now let us finish Theorem \ref{thm:growth along path}.
\begin{proof}[Proof of Theorem \ref{thm:growth along path}]
Fix a sequence $(a_n)_{n \geq 1}$ of positive integers.
For $\omega \in \{0,1,2\}^*$, put
\begin{equation*}
 \cW(\omega)=\{\rho \in \cW^r_\# \mid \#\Fix_a(\rho^\omega)
 \geq |\omega|\cdot a_{|\omega|}\}.
\end{equation*}
Since attracting periodic points persistent under small perturbations,
 $\cW(\omega)$ is an open subset of $\cW^r_\#$.
By Lemma~\ref{lemma:generic path}, it is sufficient to show that
 $\bigcup_{\eta \in \{0,1,2\}^*}\cW(\eta\omega)$
 is a dense subset of $\cW^r_\#$.
In other words, our goal is to show that
 for any given $\rho =(f_0, f_1, f_2) \in \cW^r_\#$, $\omega_0 \in \{0,1,2\}^*$,
 and any neighborhood $\cU \subset \Diff^r([0,1])$ of the identity map,
 there exist $\eta \in \{0,1,2\}^*$ and $h \in \cU$ such that
 $\#\Fix_a(\rho_h^{\eta\omega_0}) \geq n\cdot a_n$,
 where $n=|\eta\omega_0|$.
%
%
%
%
%

First, by Proposition \ref{prop:flat periodic points},
 after a small perturbation of $f_0$ if necessary,
 we may assume that there exist $p_* \in \Int J$
 and $\gamma \in \{0,1,2\}^*$ such that
 $p_*$ is an $r$-flat $0\gamma$-periodic point with
 $p_* \notin \Sigma^\gamma(p_*)$.
Now we choose $x_0 \in J \cap f_0^{-1}(J)$
such that $x_0$ is $(f_1, f_2)$-generic and
$x_0 \not\in \mathcal{O}_{-}(p_{\ast}, \rho)$.
Such $x_0$ exists since these
are generic conditions in $J \cap f_0^{-1}(J)$.
We put $x_1 = \rho^{\omega_0}(x_0)$.
Since $\rho \in \cW^r_\#$, we can
take $\omega_1 \in \{ 0, 1, 2\}^{\ast}$
such that $x_2 = \rho^{\omega_1\omega_0}(x_0) \in \mathrm{Int}(J)$.

We choose two disjoint non-empty open intervals $U$ and $V$
in $\mathrm{Int}(J) \cap \mathrm{Int}(f_0(J))$such that
\begin{itemize}
\item $V$ is a neighborhood of $p_{\ast}$, and
\item $U$, $V$ are so small that both are disjoint from
$\Sigma^{\omega_1 \omega_0}(x_0) \cup \Sigma^{\gamma}(p_{\ast})$.
\end{itemize}

Now we apply Lemma~\ref{lemma:connecting 2}
(viewing $(x_2, x_0, p_{\ast})$ as $(p, p', \hat{p})$)
to obtain $\omega_2$ and $h_1 \in \mathcal{U}$
such that $\supp (h_1) \subset U \cup V$,
$\rho_{h_1}^{\omega_2}(x_2) = x_0$, and
$[\rho^{\omega_2}_{h_1}(t)]_{x_2}
=[\rho^{\omega_1\omega_0}(t)]^{-1} + o(t^r)$.

By Remark~\ref{rmk:solo2},
we can choose $y \in U$ which appears in $\Sigma_{h_1}^{\omega_2}(x_2)$ only once.
We take the word $\omega'_2$ of the form $\omega_2|_k$
for some $k$ such that
$y= \rho_{h_1}^{\omega'_2}(x_2) \in U$.
Since $U$ is disjoint from
$\Sigma^{\omega_1 \omega_0}(x_0) \cup \Sigma^{\gamma}(p_{\ast})$, we have
$y \not\in\Sigma_{h_1}^{\omega'_2\omega_1\omega_0 \omega''_2}(y)$, where $\omega''_2$ is the (unique) word which satisfies
$\omega_2 = \omega''_2\omega'_2$.
Notice that, by construction, we also know that
$y$ is an $r$-flat
$(\omega'_2\omega_1\omega_0 \omega''_2)$-periodic point of
$\rho_{h_1}$.

Then, as in the proof of Theorem~\ref{thm:arbitrary growth}
in Section~\ref{sec:flat} (see also Remark~\ref{rem:perper}), we can find
$h_2 \in \mathrm{Diff}^{\infty}([0, 1])$ which is sufficiently close to
the identity in $C^r$ and supported in an arbitrarily small neighborhood of $y$
such that $\mathrm{Fix}_a(\rho_{h_2 \circ h_1}^{\omega'_2\omega_1\omega_0 \omega''_2}) \geq |\omega'_2\omega_1\omega_0 \omega''_2|\cdot a_{|\omega'_2\omega_1\omega_0 \omega''_2|}$.

Since each generator of $\rho_{h_2\circ h_1}$ is a diffeomorphisms on its image,
we have the same estimate for
$\mathrm{Fix}_a(\rho_{h_2 \circ h_1}^{\omega''_2\omega'_2\omega_1\omega_0 })$.
Thus letting $\eta = \omega''_2\omega'_2\omega_1 = \omega_2\omega_1$ and $h=h_2 \circ h_1$,
we complete the proof.

\end{proof}


\section{Robustness of sign condition}\label{sec:sign condition}

In this section we discuss the stability of the sign condition
which is posed on $\cW^r$ (see Section~\ref{sec:T}).
Let $f$ be an element of $\cE^2$,
$(p, q)$ be a repeller-attractor pair of $f$, and
$z_0$ be a heteroclinic point of it.
If $\bar{f}\in\cE^2$ is sufficiently
$C^1$-close to $f$, we can find the continuations
$p_{\bar{f}}$, $q_{\bar{f}}$ of $p$, $q$, respectively, and
$(p_{\bar{f}}, q_{\bar{f}})$ will be a repeller-attractor pair for $\bar f$.
Furthermore, $z_0$ will remain
 a heteroclinic point of $(p_{\bar{f}}, q_{\bar{f}})$
 and thus the sign of $z_0$ for $\bar{f}$ is well-defined.

Let $\varphi_{\bar{f}}:W^u(p_{\bar{f}},\bar{f}) \ra \RR$ 
 and $\psi_{\bar{f}}:W^s(p_{\bar{f}},\bar{f})$ be $C^r$ linearizations
 at $p_{\bar{f}}$ and $q_{\bar{f}}$ such that
 $\varphi_{\bar{f}}'(p_{\bar{f}})= \psi_{\bar{f}}'(q_{\bar{f}})=1$.
\begin{prop}\label{prop:conti}
In the above setting,
$A(\psi_{\bar{f}} \circ \varphi_{\bar{f}}^{-1})_{\varphi_{\bar{f}}(z_0)}$
 and
$S(\psi_{\bar{f}} \circ \varphi_{\bar{f}}^{-1})_{\varphi_{\bar{f}}(z_0)}$
 (if $r\geq 3$) vary continuously with respect to $\bar{f}$
 in the $C^2$- (resp. $C^3$-) topology.
\end{prop}

Proposition~\ref{prop:conti} immediately implies the following:
\begin{prop}
\label{prop:sign}
Under the above setting, if $\tau_A(z_0,f) \neq 0$,
then $\tau_A(z_0,f)=\tau_A(z_0,\bar{f})$ for any $\bar{f} \in \cE^2$ sufficiently $C^2$-close to $f$.
The same stability property holds for $\tau_S(z_0,f)$ in the $C^3$-topology.
\end{prop}

As a result, we see that the sign condition in the definition of $\cW^r$ is $C^r$-open
for any $r \geq 2$.

Let us start the proof of Proposition~\ref{prop:conti}.

\begin{proof}
We give the proof only for the continuity of
$A(\psi_{\bar{f}} \circ \varphi_{\bar{f}}^{-1})_{\varphi_{\bar{f}}(z_0)}$
Since the proof for
$S(\psi_{\bar{f}} \circ \varphi_{\bar{f}}^{-1})_{\varphi_{\bar{f}}(z_0)}$
 follows the same lines, we leave it to the reader.

Let $\vphi:W^u(p,f) \ra \RR$ and $\psi: W^s(q,f) \ra \RR$ be the normalized $C^2$-linearizations of $f$ at $p$ and $q$,
respectively. Put $\lambda_p=f'(p)$, $\lambda_q=f'(q)$,$x_0=\vphi(z_0)$,
$y_n=\psi \circ f^n(z_0)$  and $z_n=f^n(z_0)$ for $n \in \ZZ$.
Notice that we have the following:
\begin{equation*}
 f^{2n}= f^n \circ \psi^{-1} \circ \psi
 \circ \varphi^{-1} \circ \varphi \circ f^n
 = \left( \psi^{-1} \circ (L_{\lambda_q})^n \right)
 \circ (\psi \circ \varphi^{-1}) \circ
 \left((L_{\lambda_p})^n \circ \varphi \right).
\end{equation*}
For the second equality we used the fact that $\varphi$ and $\psi$ are linearizations.
Then by applying the cocycle property of $A$ to this equality, we obtain
\begin{align*}
\lambda_p^{-n} \cdot A([f^{2n}]_{z_{-n}})
 =  \lambda_q^n \cdot A([\psi^{-1}]_{y_n})
  \cdot (\psi \circ \vphi^{-1})'(x_0) \cdot \varphi'(z_{-n})\quad \\
 + A([\psi \circ \vphi^{-1}]_{x_0}) \cdot \vphi'(z_{-n})
 + \lambda_p^{-n} \cdot A([\vphi]_{z_{-n}}).
\end{align*}
Since $\lim_{n \to + \infty}\vphi'(z_{-n})  =\vphi'(p)=1$,
 $\lim_{n \to + \infty} y_n  = \psi(q)= 0$,
 and $\lambda_p >1 >\lambda_q>0$,
 we have
\begin{equation*}
A([\psi^{-1} \circ \vphi]_{x_0})
 =\lim_{n \ra +\infty}  \lambda_p^{-n} \cdot A([f^{2n}]_{z_{-n}})
\end{equation*}
On the other hand, the value
$\lambda_p^{-n} \cdot A([f^{2n}]_{z_{-n}}) $ can be
calculated by applying
the cocycle property of $A$ to $f^{2n}$ directly. Then we obtain
\begin{equation}
\label{eqn:A}
\lambda_p^{-n} \cdot A([f^{2n}]_{z_{-n}})
  = \sum_{m=-n}^{n-1} A([f]_{z_m}) \cdot \lambda_p^{-n}
 \cdot (f^{m+n})'(z_{-n}).
\end{equation}
Notice that the term $A([f]_{z_m})$ has the absolute value bounded
from above by a constant independent of $m$.
So let us analyze the term $\lambda_p^{-n} \cdot (f^{m+n})'(z_{-n})$
Since $\lambda_p = f'(p)$, we have
\begin{equation}\label{eqn:bdd}
 \left| \log (\lambda_p^{-n}(f^n)'(z_{-n}))\right|
  \leq \sum_{j=0}^{n-1}\left|\log f'(z_{-j})-\log f'(p)\right|
\end{equation}
Since $|z_{-n}-p|$ converges to $0$ exponentially as $n$ goes
to $\infty$, and the function $\log (f')$ is Lipschitz, we can deduce
that there exists a constant $C_1$ independent of $n$ such
that the last term in (\ref{eqn:bdd}) is bounded by $C_1$ from the above.

Take $\alpha \in (\max\{\lambda_q, \lambda_p^{-1}\},1)$
and fix $N >1$  sufficiently large so that for every $m$ satisfying
$|m|\geq N$, the inequality $f^m(z_0) < \alpha^{|m|-N}$ holds.
Such $N$ does exist since $z_0$ is a heteroclinic point.
Then we put $C_2=\max\{(f^N)'(z_0), (f^{-N})'(z_0)\}$.

Now, for any $m \in \ZZ$ with $|m| \geq N$, we have
\begin{equation}\label{eqn:exp}
 \lambda_p^{-n}\cdot (f^{m+n})'(z_{-n})
  = (f^m)'(z_0) \cdot \lambda_p^{-n} (f^n)'(z_{-n})
  <\alpha^{|m|-N} \cdot C_2 \cdot e^{C_1}.
\end{equation}
Notice that the choice of $\alpha$, $N$ and $C_2$ is independent of $n$,
 and the inequality (\ref{eqn:exp}) holds for every $\bar{f}$ 
 which is sufficiently $C^2$-close to $f$.
This implies that the infinite series (\ref{eqn:A}) converges
 exponentially and the convergence of
 $\lambda_p^{-n} A([f^{2n}]_{z_{-n}})$
 to $A([\psi \circ \vphi^{-1}]_{x_0})$ is uniform
 with respect to $C^2$-perturbation of $f$.
 This implies that $A([\psi \circ \vphi^{-1}]_{x_0})$
is a continuous function of $f$ in the $C^2$-topology.
\end{proof}

One can see from the proof that
 the pairs $(A([f^{2n}]_{f^{-n}(z)}),S([f^{2n}]_{f^{-n}(z)}))$
 and $(A[\psi \circ \vphi^{-1}]_\vphi(z),S[\psi \circ \vphi^{-1}]_\vphi(z))$
 have the same sign for any sufficiently large $n$.
This, together with the cocycle properties of $A$ and $S$,
 implies the following sufficient condition to
 determine $\tau_A(z,f)$ and $\tau_S(z,f)$.
\begin{cor}
\label{cor:sign}
Let $f$ be a map in $\cE^3$, $(p,q)$ be a repeller-attractor pair of $f$,
 and $\tau_1,\tau_2 \in \{\pm 1\}$.
If $\sgn(A(f)_z)=\tau_1$ and $\sgn(S(f)_z)=\tau_2$
 for all $z \in W^u(p) \cap W^s(q)$,
 then $\tau_A(z,f)=\tau_1$ and $\tau_S(z,f)=\tau_2$
 for any heteroclinic point $z$ of $(p,q)$.
\end{cor}

\section{A criterion for sign conditions}
\label{sec:criterion}
In this section, we give the following simple criterion
 for sign conditions,
 which is given only interms of positions and multipliers of repellers
 and attractors.
\begin{prop}
\label{prop:criterion 1}
If $f \in \cE^3$ has three fixed points $q_0<p<q_1$
 such that $f'(q_0)<f'(p)^{-1}<f'(q_1)<1<f'(p)$
 and $W^u(p,f)=(q_0,q_1)$,
 then $f$ satisfies Sign conditions I and II.
\end{prop}
The following criterion for existence of a persistent blender
 is shown in {\cite[Example 1]{Sh}.
\begin{prop}
\label{prop:blender}
 Suppose that $(f_1,f_2) \in \cA^1(\{1,2\})$ satisfies
 that $f_1'<1$ and $f_2'<1$ on an closed interval $[a,b] \subset (0,1)$,
 $f_1(a)=a$, $f_2(b)=b$, and $f_1(b)>f_2(a)$.
 Then $(f_1,f_2)$ is a $C^1$-persistent blender for any closed interval
 $J \subset (a,b)$.
\end{prop}

Let us give a simple example of semigroups
 in $\cW^\infty$ by using these propositions.
Fix real numbers $p,q_0,q_1,r,\delta$ such that
 $0<q_0<p<q_1<1<r$, $r-p<q_1-q_0$, and $0<2\delta<\min\{q_0,1-q_1\}$.
Let $\rho=(f_0,f_1,f_2) \in \cA^\infty(\{0,1,2\})$
 be a triple given by
\begin{align*}
 f_0(x) & = x+\epsilon(x-q_0)(x-p)(x-q_1)(x-r), \\
 f_1(x) & = (1-\delta)x+\delta^2,\\
 f_2(x) & = (1-\delta)x+\delta(1-\delta).
\end{align*}
 with $\epsilon>0$.
A direct computation shows that
 $(p,q_0)$, $(p,q_1)$ are repeller-attractor pairs of $f_0$,
 $\delta$ is the unique fixed point of $f_1$,
 and $1-\delta$ is the unique fixed point of $f_2$.
Set $J=[2\delta,1-2\delta]$.
The interval $\Int J$ contains the fixed points $p,q_0,q_1$ of $f_0$
 and all heteroclinic points between them.
Since $0<2\delta<1$, we have $f_1(1-\delta)>f_2(\delta)$.
By Proposition {\ref{prop:blender}},
 the pair $(f_1,f_2)$ is a $C^1$-persistent blender on $J$.
Since $q_0<p<q_1<r$, $r-p<q_1-q_0$, and
\begin{align*}
 f_0'(p) \cdot f_0'(q_i) &
 = 1-\epsilon(p-q_i)^2\{r-p+(-1)^i(q_1-q_0)\}+O(\epsilon^2)
\end{align*}
 for each $i=0,1$,
 we have $f_0'(p) \cdot f_0'(q_0)<1<f_0'(p) \cdot f_0'(q_1)$
 if $\epsilon>0$ is sufficiently small.
By Proposition \ref{prop:criterion 1},
 $f_0$ satisfies Sign conditions I and II.
Therefore, $\rho=(f_0,f_1,f_2)$ is an element of $\cW^\infty$
 when $\epsilon$ is sufficiently small.
We also see that the set $\cW^\infty_\# \cap \cW^\infty_{\mathrm{att}}$
 defined in Section {\ref{sec:mainre}} is non-empty.
Indeed, if $\delta>0$ is sufficiently small,
 then $f_1^N([0,\delta]) \cup f_2^N([1-\delta,1]) \subset J$
 for some large $N$.
This implies that $\rho$ is an element of $\cW^\infty_\#$.
It is easy to check that $\rho$ is an element of $\cW^\infty_{\mathrm{att}}$
 if $\epsilon$ is sufficiently smaller than $\delta$.

Proposition {\ref{prop:criterion 1}}
 is a direct consequence of the following
\begin{prop}
\label{prop:criterion} 
The following hold
 for $f \in \cE^r$ with $r \geq 2$
 and its repeller-attractor pair $(p,q)$:
\begin{enumerate}
\item There exists $z_* \in (p,q)$ such that
$\tau_A(z_*,f)=\sgn(p-q)$.
\item If $r \geq 3$ and $f'(p)f'(q) \neq 1$,
 then there exists $z_\# \in (p,q)$ such that
\begin{equation*}
\tau_S(z_\#,f)=\sgn(f'(p)f'(q)-1). 
\end{equation*}
\end{enumerate}
\end{prop}
We reduce the proposition to the following
\begin{lemma}
\label{lemma:criterion} 
Let $F$ be a $C^r$ map from $\RR_+=\{x \in \RR \mid x>0\}$
 to $\RR_-=\{x \in \RR \mid x<0\}$ with $r \geq 2$.
Suppose that $F'>0$
 and there exist positive real numbers $\lambda$ and $\mu$
 such that $\mu<1<\lambda$
 and $F(\lambda^n)= - \mu^n$ for all $n \in \ZZ$.
Then, the following hold:
\begin{enumerate}
\item There exists $x_* \in \RR_+$ such that
 $A(F)_{x_*}<0$.
\item If $r \geq 3$ and $\lambda\mu \neq 1$,
 then there exists $x_\# \in \RR^+$
 such that $\sgn(S(F)_{x_\#})=\sgn(\lambda \mu-1)$.
\end{enumerate}
\end{lemma}
\begin{proof}
[Proof of Proposition \ref{prop:criterion}
 from Lemma \ref{lemma:criterion}]
Let $\varphi:W^u(p,f) \ra \RR$ 
 and $\psi:W^s(q,f) \ra \RR$ be the linearizations of $f$
 at $p$ and $q$ such that $\varphi'(p)=\psi'(q)=1$.
Set $I_p=\varphi(W^u(p,f) \cap W^s(q,f))$,
 $I_q=\psi(W^u(p,f) \cap W^s(q,f))$,
 and $H=\psi \circ \varphi^{-1}$.

First, we suppose that $p<q$.
Then, $I_p=\RR_+$, $I_q=\RR_-$, and
\begin{equation*}
 H(\lambda^n) = \psi \circ f^n \circ \varphi^{-1}(1) = \mu^n \cdot H(1)
\end{equation*}
 for any $n \in \ZZ$.
Set $F(x)=H(x)/|H(1)|$.
Then, the map $F$ satisfies the assumption of Lemma \ref{lemma:criterion}.
Since $A(H)_x=A(F)_x$ and $S(H)_x=S(F)_x$,
 by applying Lemma \ref{lemma:criterion},
 we obtain the proposition for this case.
The proof for the case $q>p$ can be obtained in a similar way
 with $F(x)=-H(x)/H(-1)$.
\end{proof}

\begin{proof}
[Proof of Lemma \ref{lemma:criterion}] 
By the mean value theorem,
 there exists $x_n \in (\lambda^n,\lambda^{n+1})$ such that
\begin{equation*}
 F'(x_n)= \frac{F(\lambda^{n+1})-F(\lambda^n)}{\lambda^{n+1}-\lambda^n}
 =\left(\frac{\mu}{\lambda}\right)^n\frac{1-\mu}{\lambda-1}
\end{equation*}
 for any $n \in \ZZ$.
Since $\lambda>1>\mu>0$, we have $F'(x_{n+1})<F'(x_n)$ for any $n$.
This implies that there exists $x_* \in (x_0,x_1)$
 such that $F''(x_*)<0$.

Set $G=(F')^{-\frac{1}{2}}$. By a direct calculation, we have
\begin{equation}
\label{eqn:Schwarzian}
 G''=-\frac{1}{2}(F')^{-\frac{1}{2}} \cdot S(F).
\end{equation}
Since $\lambda^n<x_n<\lambda^{n+1}$, we have
\begin{equation*}
\lambda^2-\lambda< \frac{x_{n+2}-x_n}{\lambda^n}< \lambda^3-1.
\end{equation*}
We also have
\begin{align*}
\frac{G(x_{n+2})-G(x_n)}{x_{n+2}-x_n} 
 & = \frac{\left(\left(\frac{\lambda}{\mu}\right)^{\frac{n+2}{2}}
   -\left(\frac{\lambda}{\mu}\right)^\frac{n}{2}\right)
  \cdot  \sqrt{\frac{\lambda-1}{1-\mu}}}{x_{n+2}-x_n}\\
 & = (\lambda \mu)^{-\frac{n}{2}}
 \cdot \left(\frac{\lambda}{\mu}-1\right)
 \cdot \sqrt{\frac{\lambda-1}{1-\mu}}
 \cdot \frac{\lambda^n}{x_{n+2}-x_n}.
\end{align*}
Hence, we can choose a constant $C>0$
 such that
\begin{equation*}
 C^{-1}(\lambda\mu)^{-\frac{n}{2}}
 < \frac{G(x_{n+2})-G(x_n)}{x_{n+2}-x_n} 
 < C(\lambda\mu)^{-\frac{n}{2}}
\end{equation*}
 for any $n \in \ZZ$.

By the mean value theorem, there exists $y_n \in (x_n,x_{n+2})$
 such that $G'(y_n)=(G(x_{n+2})-G(x_n))/(x_{n+2}-x_n)$.
Notice that $\lim_{n \ra +\infty} y_n=+\infty$
 and $\lim_{n \ra -\infty}y_n=0$.
If $\lambda\mu>1$, then
 $\lim_{n \ra +\infty}G'(y_n)=0$ and
 $\lim_{n \ra -\infty}G'(y_n)=+\infty$.
This implies that there exists $x_\# \in \RR_+$
 such that $G''(x_\#)<0$, and hence, $S(F)_{x_\#}>0$
 by Equation (\ref{eqn:Schwarzian}).
If $\lambda\mu<1$, then
 $\lim_{n \ra +\infty}G'(y_n)=+\infty$ and
 $\lim_{n \ra -\infty}G'(y_n)=0$.
This implies that there exists $x'_\# \in \RR_+$
 such that $G''(x_\#)>0$, and hence, $S(F)_{x'_\#}<0$.
\end{proof}

\vspace{1cm}

\begin{itemize}
\item[]  \emph{Masayuki Asaoka (asaoka@math.kyoto-u.ac.jp)}
\begin{itemize}
\item[] Department of Mathematics, Kyoto University,
\item[] Kitashirakawa-Oiwakecho, Kyoto, Japan
\end{itemize}
\item[] \emph{Katsutoshi Shinohara (herrsinon@07.alumni.u-tokyo.ac.jp)}
\begin{itemize}
\item[]  Department of Mathematical Sciences,
\item[]  Tokyo Metropolitan University, 1-1 Minami-Osawa, Hachioji
\item[]  Tokyo, Japan
\end{itemize}
\item[] \emph{Dmitry Turaev (d.turaev@imperial.ac.uk)}
\begin{itemize}
\item[] Department of Mathematics, Imperial College London
\item[] 180 Queen's Gate, London, United Kingdom,
\item[] and Lobachevsky University of Nizhny Novgorod, 603950 Russia.
\end{itemize}
\end{itemize}

\end{document}